\documentclass[11pt]{article}

\usepackage[top=3cm,bottom=3cm,left=3.3cm,right=3.3cm]{geometry}

\RequirePackage[utf8]{inputenc} 

\usepackage{mathtools}
\usepackage{psfrag}
\usepackage{tikz} 
\usetikzlibrary{decorations.pathmorphing, patterns,shapes}
\usepackage{float}
\usepackage{amsmath}
\usepackage{amsthm}
\usepackage{acronym}
\usepackage{color}
\usepackage{mathrsfs} 
\usepackage{enumitem}
\usepackage{bbm}

\usepackage{hyperref}

\usepackage{soul}

\usepackage{graphicx}

\usepackage{float}
\usepackage{epstopdf}
\usepackage{subfigure}
\usepackage{longtable}
\usepackage{multicol}

\usepackage[square, numbers, comma, sort&compress]{natbib}

\usepackage{amsmath,amsfonts,amssymb,amscd,amsthm,xspace}
\theoremstyle{plain}

\newtheorem{theorem}{Theorem}[section]
\newtheorem{corollary}[theorem]{Corollary}
\newtheorem{lemma}[theorem]{Lemma}
\newtheorem{Prop}[theorem]{Proposition}

\theoremstyle{definition}

\theoremstyle{remark}
\newtheorem{remark}[theorem]{Remark}

\usepackage{mathrsfs}

\newcommand{\norm}[1]{\left\lVert#1\right\rVert}
\newcommand{\ff}[1]{f\left( #1  \right) }

\newcommand*\diff{\mathop{}\!\mathrm{d}}

\newcommand{\normf}[1]{\left\lVert#1\right\rVert_{\text{flat}}}
\newcommand{\R}{{\mathbb{R}}}

\newcommand{\Z}{{\mathbb{Z}}}
\newcommand{\N}{{\mathbb{N}}}
\newcommand\1{\mathbbm{1}}

\newcommand{\ue}{u_{\varepsilon}}

\newcommand{\tgae}{\tilde{\gamma}_{\varepsilon}}

\newcommand{\go}{\gamma^{\textup{opt}}}

\newcommand{\gae}{\gamma_{\varepsilon}}

\newcommand{\ggae}{g_{\varepsilon}}

\newcommand{\W}{{W^{1,1}}(a,b)}

\newcommand{\WMQ}{{W^{-1,q}}(a,b)}

\newcommand{\WMQO}{{W^{-1,q}}}
\newcommand{\LE}{{L^{1}}(a,b)}

\newcommand{\LU}{{L^{\infty}}(a,b)}

\newcommand{\M}{\mathcal{M}(a,b)}

\newcommand{\Ge}{\mathscr{G}_{\varepsilon}}

\newcommand{\Gee}{\mathscr{G}^{\prime}_{\varepsilon}}

\newcommand{\dd}{ \, \textup{d}}

\newcommand{\bue}{\bar{u}_{\varepsilon}}

\DeclareMathOperator{\sign}{sign}

\newcommand{\mrs}{\mathbin{\vrule height 1.6ex depth 0pt width
0.13ex\vrule height 0.13ex depth 0pt width 1.3ex}}

\def\Xint#1{\mathchoice
{\XXint\displaystyle\textstyle{#1}}%
{\XXint\textstyle\scriptstyle{#1}}%
{\XXint\scriptstyle\scriptscriptstyle{#1}}%
{\XXint\scriptscriptstyle\scriptscriptstyle{#1}}%
\!\int}
\def\XXint#1#2#3{{\setbox0=\hbox{$#1{#2#3}{\int}$ }
\vcenter{\hbox{$#2#3$ }}\kern-.6\wd0}}

\def\dashint{\Xint-}

\allowdisplaybreaks[3]

\begin{document}

\begin{center}
\begin{Large}
{\bf Eigendamage: an Eigendeformation model \\ 
for the variational approximation of \\ 
cohesive fracture -- a one-dimensional case study

}
\end{Large}
\end{center}

\begin{center}
\begin{large}
\renewcommand{\thefootnote}{\fnsymbol{footnote}}
Veronika Auer-Volkmann\footnote{Universit{\"a}t Augsburg, Institut f{\"u}r Mathematik, 
86135 Augsburg, Germany. \label{uni-adr}}\renewcommand{\thefootnote}{\arabic{footnote}}\addtocounter{footnote}{-1}\footnote{{\tt veronika.auervolkmann@gmail.com}}, 
Lisa Beck\textsuperscript{\ref{uni-adr}}\footnote{\tt lisa.beck@math.uni-augsburg.de}, and 
Bernd Schmidt\textsuperscript{\ref{uni-adr}} \!\!\footnote{\tt bernd.schmidt@math.uni-augsburg.de}
\end{large}
\end{center}

\begin{center}
\today
\end{center}
\smallskip

\begin{abstract}
We study an approximation scheme for a variational theory of cohesive fracture in a one-dimensional setting. Here, the energy functional is approximated by a family of functionals depending on a small parameter $0 < \varepsilon \ll 1$ and on two fields: the elastic part of the displacement field and an eigendeformation field that describes the inelastic response of the material beyond the elastic regime. We measure the inelastic contributions of the latter in terms of a non-local energy functional. Our main result shows that, as $\varepsilon \to 0$, the approximate functionals $\Gamma$-converge to a cohesive zone model. 
\end{abstract}
\smallskip

\begin{small}


\end{small}

\section{Introduction} 

A tension test on a bar will typically show that small deformations are completely reversible (elastic regime) while large deformations lead to complete failure (fracture regime). Only for very brittle materials one observes a sharp transition between these two regimes (brittle fracture). By way of contrast, many materials exhibit an intermediate {\em cohesive zone \textup{(}damage regime\textup{)}} in which plastic flow occurs and a body shows gradually increasing damage before eventual rupture (ductile fracture). 

Variational models have been extremely successfully applied to problems in fracture mechanics, cf., e.g., \cite{Braides:98,afp,BourdinFrancfortMarigo:08,DelPiero:13} and the references therein including, in particular, the seminal contribution of Francfort and Marigo \cite{FrancfortMarigo:98}. Here energy functionals are considered that act on deformations in the class of functions of bounded variation (or deformation). The derivatives of these functions are merely measures, and the singular part of such a measure is directly related to the inelastic behavior of the bar. The resulting variational problems are of {\em free discontinuity type} allowing for solutions with jump discontinuity (macroscopic cracks). Moreover, within a damage regime the strain can contain a diffuse singular part describing continuous deformations beyond the elastic regime that can be related to the occurrence of microcracks, see e.g.\ \cite{FrancfortMarigo:93,FremondNedjar:96,FrancfortGarroni:06,GarroniLarsen:09,DelPieroTruskinovsky:09,CagnettiToader:11,PhamMarigo:13,AlessiMarigoVidoli:14}. In fact, when considering variational problems with stored energy functions of linear growth at infinity and surface energy contributions that scale linearly for small crack openings, all these contributions to the total strain interact, cp.\ \cite{BouchitteBraidesButtazzo:95}, which renders the problem challenging, both from a theoretical and a computational point of view. 

As free discontinuity problems are of great interest not only in fracture mechanics but also in image processing, several approximation schemes have been proposed with the aim to devise efficient numerical approaches to simulations. Most notably, the Ambrosio--Tortorelli approximation \cite{AT92,AT90} has triggered a still continuing interest in {\em phase field models} in which a second field (the `phase field') is introduced that can be interpreted as a damage indicator and whose value influences the elastic response of the material. 

With a particular focus on cohesive zone damage models we refer to, e.g., \cite{DalMasoIurlano:13,Iurlano:13,ContiFocardiIurlano:16,DalMasoOrlandoToader:16,CrismaleLazzaroniOrlando:18}. A small parameter $\varepsilon$ is introduced in such models that corresponds to an intrinsic length scale over which sharp interfaces of the phase field variable are smeared out. A different approach has been initiated by Braides and Dal Maso for the Mumford--Shah functional, and then extended to various generalized settings in, e.g., \cite{BrGa98,Braides:98,LuVi97,CortesaniToader:99b,LVDD07,Neg06,Lussardi:08,LM13} which involves a non-local approximation of the original field $u$ in terms of convolution kernels with intrinsic length scale $\varepsilon \ll 1$. 

Our main motivation comes from the Eigenfracture approach to brittle materials that has been developed in \cite{sfo} and further considered in \cite{PWO:21,PandolfiOrtiz:12,QinamiPandolfiKaliske:20}. Our main aim is to extend this model to a ductile fracture regime with a significant damage zone. The variables of the model are the deformation field $u_\varepsilon$ and an {\em eigendeformation} field $g_\varepsilon$, which induces a decomposition of the strain $u_\varepsilon' = (u_\varepsilon' - g_\varepsilon) + g_\varepsilon$ into an elastic and an inelastic part, the latter describing deformation modes that cost no local elastic energy. (We refer to \cite{M13} for more details on the concept of eigendeformations to describe inelastic deformations and, in particular, plastic deformations.) The energy associated to the formation and increase of damage is accordingly modeled in terms of a non-local functional acting on $g_\varepsilon$, which replaces the non-local contribution defined in terms of a simple $\varepsilon$-neighborhood of the crack set in the original Eigenfracture model by  a more general (and softer) convolution approximation. 

We would like to point out that our set-up thus introduces a novel modeling aspect to damage functionals. Instead of an explicit dependence of the stored energy function on the damage as being encoded in a phase field, in our model the constitutive laws, i.e., the linear elastic energy $|\cdot|^2$ and fracture contribution $f$ (see below) remain unchanged. An increase of damage is rather related to a transition from the elastic deformation field to the eigendeformation field. In particular, plastic deformations at the onset of the inelastic regime need not immediately lead to softening of the material. With respect to non-local convolution approximation schemes of the deformation field $u$ we remark that in our model such non-local contributions need to be evaluated only near the support of $g_\varepsilon$ but not on purely elastic regions.  

In the present contribution we focus on the one-dimensional case. In this setting our analysis will benefit from the corresponding studys~\cite{LuVi97} of Lussardi and Vitali for pure convolution functional. Indeed, we will follow along the same path in order to adapt and extend their methods to our two field set-up. There are, however, a number of notable differences in our analysis which lead us, also in view of later extensions to higher dimensions~\cite{AuerBeckSchmidt:21b}, to provide a self-contained account of our results. A main difficulty stems from the fact that there is no pre-assigned functional relation between the eigendeformation fields $g_\varepsilon$ and the strain fields $u_\varepsilon'$. Rather these quantities are merely `coupled by regularity' in the sense that $u_\varepsilon' - g_\varepsilon \in L^2$. As the limit of $g_\varepsilon$ needs to be studied in a rather weak space, this leads to technical difficulties when transferring asymptotic properties from $u_\varepsilon$ to $g_\varepsilon$. 

Our results also constitute the first step towards higher-dimensional models. In particular, the case of antiplane shear will be addressed in a forthcoming contribution \cite{AuerBeckSchmidt:21b}. Here the lack of a direct relation between $g_\varepsilon$ and $\nabla u_\varepsilon$ and hence the absence of an underlying gradient structure will pose severe additional challenges. 

\subsection*{Outline} 
We start by describing the setting of the problem and by stating the main results in Section~\ref{sec:settingoned}. In Section~\ref{sec_preliminaries} we remind some facts on functions of bounded variation and the flat topology. Section~\ref{sec: compactness1d} is devoted to a compactness result. The $\Gamma$-lower limit for the eigendamage model is established in Section~\ref{sec: estimate from below 1d}. To this end, we first derive the estimate from below of the jump part, subsequently the estimate from below of the volume term and the Cantor term, and the proof of the $\Gamma\text{-}\liminf$ inequality is then completed  by combining the previous results. In Section~\ref{sec:upper-limit} we then address the estimate from above of the $\Gamma$-upper limit. Finally, in Section~\ref{sec: minimal energies 1d} the asymptotic behavior of the minimal energies with respect to the eigendeformation variable is studied.

\section{Setting of the problem and main result} \label{sec:settingoned}

Suppose that a beam occupies the region $(a,b)$ with $0<a<b<\infty$ and that a displacement $u \colon (a,b) \to \R$ affects the beam. As cohesive energy associated with $u$ we shall consider
\begin{align}
F(u) &= \int_a^b { \psi(|u^{\prime}|) } \dd x+ 2 \sum_{x \in J_u} \ff{\frac{1}{2}|[u](x)|} + c_0 |D^cu|(a,b), \label{GF1d}
\end{align}
where $c_0$ is a fixed positive constant, and $\psi, f \colon [0,\infty) \to [0,\infty)$ are functions defined via

\noindent
\begin{minipage}{0.49\textwidth}
\begin{align*}
\psi(t) & = \begin{cases}
t^2 & \text{if } t < \frac{c_0}{2},\\
c_0 \, t - \frac{c_0^2}{4} & \text{if } t \geq \frac{c_0}{2},
\end{cases} \\
\intertext{and}
f(t) & =  \begin{cases}
c_0\, t & \quad   \text{if } t < 1, \\
c_0 & \quad  \text{if } t\geq 1.
\end{cases}
\end{align*}
\end{minipage} \hfill
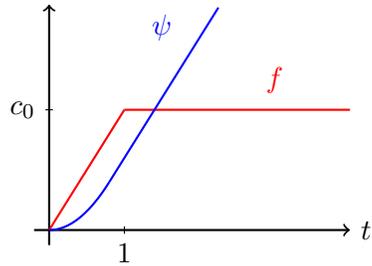
\begin{minipage}{0.49\textwidth}
 
\begin{center}
\begin{tikzpicture}
 
 \draw[thick,->] (-0.2,0) -- (4,0) node[right] {$t$};
 \draw[thick,->] (0,-0.2) -- (0,3);
 
 \draw (1,0.05) -- (1,-0.05) node[below] {$1$};
 \draw (0.05,1.6) -- (-0.05,1.6) node[left] {$c_0$};
 
 \draw[thick,red] (0,0) -- (1,1.6);
 \draw[thick,red] (1,1.6) -- (4,1.6);
 
 \draw[domain=0:0.8,smooth,variable=\x,blue,thick] plot (\x,{\x*\x});
 \draw[domain=0.8:2.25,smooth,variable=\x,blue,thick] plot (\x,{1.6*\x-0.64});
 
 \node[red] at (3,2) {$f$};
 \node[blue] at (1.5,2.7) {$\psi$};
 
\end{tikzpicture}
\end{center}
 
\end{minipage}

\vskip.4\baselineskip

\noindent Note that~$f$ is the simplest continuous function such that $f(0)=0$, 
\begin{equation*}
\lim_{t \to 0^+} \frac{f(t)}{t}=c_0  \quad  \text{and} \quad \lim_{t \to \infty} f(t)=c_0. 
\end{equation*}

The main ingredients in the energy~\eqref{GF1d} are a volume term, depending on the strain of the beam $u^{\prime}$  and corresponding to the stored energy, a surface term, depending on the crack opening~$[u]  \coloneqq  u(\cdot \, +) - u(\cdot \, -)$ on the jump set $J_u$ and modeling the energy caused by cracks, and finally a diffuse damage term, depending on the Cantor derivative $D^cu$ and corresponding to the energy caused by microcracks.

The natural function space in order to study such functionals in one dimension is the space $BV(a,b)$ of functions of bounded variation on $(a,b)$. Notice that the distributional derivative of each function $u \in BV(a,b)$ allows for a decomposition $Du = u' \mathcal{L}^1 + D^s u$ into the absolutely continuous and the singular part with respect to the Lebesgue measure, and the singular part $D^s u = [u] \mathcal{H}^0 \mrs J_u + D^cu $ in turn into the jump part and the Cantor part, which we have used in~\eqref{GF1d}. We consider both models with an apriori bound $\norm{u}_{\LU} \leq K$, $K < \infty$, and unrestricted models with $K = \infty$. 

We next introduce a functional depending on two fields $u \in L^1(a,b)$ (in case $K<\infty$), respectively, $u \in L^0((a,b), \overline{\R})$ (in case $K=\infty$) and $\gamma \in \M$ with a non-local approximation of the the second variable~$\gamma$, given as 
\begin{align*}  
E_{\varepsilon}(u,\gamma)  \coloneqq  
\begin{cases}
\int_a^b { |u'- g |^2 \dd x }& \quad \text{if } u \in \W, \,  \norm{u}_{\LU} \leq K, \\ + \frac{1}{\varepsilon} \int_a^b {f \big( \varepsilon \dashint_{I_{\varepsilon}(x) \cap (a,b)}{ |g | } \dd t} \big) \dd x  & \quad \phantom{if } \gamma= g \mathcal{L}^1, \, g \in L^1(a,b), \\
 & \quad \phantom{if } \text{and }   u'-g \in  L^2(a,b),\\[0.2cm]
\infty & \quad \text{otherwise},
\end{cases}
\end{align*}
with $\varepsilon > 0$, $I_{\varepsilon}(x) \coloneqq(x-\varepsilon,x+\varepsilon)$, and either $K>0$ a fixed constant or $K=\infty$. We notice that $E_{\varepsilon}(u,\gamma)$ can only be finite if $\gamma$ is absolutely continuous with respect to the Lebesgue measure, with density in $L^1(a,b)$. In this case $u'$ represents the strain of the beam and~$\gamma$ is intended to compensate~$u'$ in regions where~$u^{\prime}$ is above a certain strain level. Hence, $u'-g$ is the elastic strain of the material, while~$\gamma$ describes the deformation of the material beyond the elastic regime, indicating that a permanent deformation is exhibited if $\gamma \neq 0$. In what follows, we are interested into the asymptotic behavior of the functionals $\{E_\varepsilon\}_{\varepsilon>0}$ as $\varepsilon \searrow 0$ (in the sense of $\Gamma$-convergence). Focussing first on the case $K<\infty$, it will be described by the energy functional~$E$ which for $(u,\gamma) \in L^1(a,b) \times \M$ is defined as
\begin{equation*}
E(u,\gamma) \coloneqq
\begin{cases}
\int_a^b {|u^{\prime}- g|^2 \dd x }+ c_0 \int_a^b {  |g| \dd x } & \quad \text{if } u \in BV(a,b), \,  \norm{u}_{\LU} \leq K, \\ 
\quad + 2 \, \sum_{x \in J_u} \ff{\frac{1}{2}|[u](x)|} & \quad \quad \gamma = D^su +g \mathcal{L}^1, \, g \in L^1(a,b),\\  \quad +c_0 |D^cu|(a,b)
& \quad \phantom{if } \text{and }  u^{\prime}- g \in L^2(a,b),\\[0.2cm] 
\infty & \quad \text{otherwise}. 
\end{cases} 
\end{equation*}
Let us notice that for a finite energy $E(u,\gamma)$, the displacement field~$u$ and the eigendeformation field~$\gamma$ need to be linked in a very particular way. The singular part~$\gamma^s$ of the measure~$\gamma$ with respect to the Lebesgue measure needs to coincide with the singular part~$D^s u$ of the distributional derivative of~$u$. The absolutely continuous part~$g \mathcal{L}^1$ of~$\gamma$ instead is not completely determined by the function~$u$, but only the integrability restriction $u^{\prime}- g \in L^2(a,b)$ is required.  A particularly interesting choice of~$g$ for a given function $u \in BV(a,b)$ constitutes the unique minimizer $g^\ast$ of the optimization problem
\begin{equation}
\label{optimization_problem}
\text{to minimize } \int_a^b {|u^{\prime}- g|^2 \dd x }+ \int_a^b { c_0 |g| \dd x } \quad \text{among all } g \in L^1(a,b).
\end{equation}
By a pointwise minimization of the integrand, the minimizer~$g^\ast$ is explicitly given as 
\begin{equation}
g^\ast = \begin{cases}u^{\prime} - \sign(u') \frac{c_0}{2} & \text{if } |u^{\prime}| > \frac{c_0}{2}, \\
0 & \text{if } |u^{\prime}| \leq \frac{c_0}{2}.
\end{cases} \label{gammaaopt}
\end{equation}
For later purposes we notice that the eigendeformation field~$\gamma$ is completely described in terms of the function~$u$ as 
\begin{equation*}
\go \coloneqq D^su + g^\ast \mathcal{L}^1.
\end{equation*}
Moreover, the corresponding energy functional $E(u,\go)$ reduces to a one-field functional depending only on the displacement~$u \in BV(a,b)$, which under the additional restriction $\norm{u}_{\LU} \leq K$ (if $K < \infty$) is precisely given by the energy $F(u)$ introduced in~\eqref{GF1d}.

In order to state our $\Gamma$-convergence result we need to endow $L^1(a,b) \times \M$ with a topology. A natural choice for the first component is the strong topology on $L^1(a,b)$. One appropriate choice for the second component is the flat topology, that is the norm topology on the dual of the space of Lipschitz continuous functions with compact support, while an alternative choice is the topology induced by suitable negative $\WMQO$-Sobolev norms, see Section~\ref{sec_preliminaries} for more details. Our main result is the following: 

\begin{theorem}\label{mainresult}
Let $L^1(a,b)$ be equipped with the strong topology and $\mathcal{M}(a,b)$ be equipped with the flat topology. Assume $K<\infty$. Then the family $\{E_{\varepsilon}\}_{\varepsilon>0}$ $\Gamma$-converges to~$E$ in $L^1(a,b) \times \mathcal{M}(a,b)$, i.e., we have
\begin{enumerate}[font=\normalfont, label=(\roman{*}), ref=(\roman{*})]
\item \emph{\textup{(}$\liminf$ inequality\textup{)}} For every sequence $\{(u_{\varepsilon},\gamma_{\varepsilon})\}_{\varepsilon}$ in $\LE \times \M$ converging to  $(u , \gamma) \in L^1(a,b) \times \M$, i.e., $u_{\varepsilon} \to u$ in $L^1(a,b)$ and $\gamma_{\varepsilon} \to \gamma$ in the flat norm, we have
\begin{equation*}
\liminf_{\varepsilon \to 0} E_{\varepsilon}(u_{\varepsilon},\gamma_{\varepsilon}) \geq E(u,\gamma).
\end{equation*} 
\item \emph{($\limsup$ inequality\textup{)}} For every $(u,\gamma) \in  L^1(a,b) \times \mathcal{M}(a,b)$ there exists a sequence $\{(u_{\varepsilon},\gamma_{\varepsilon})\}_{\varepsilon}$ in $\LE \times \M$ such that  $u_{\varepsilon} \to u $ in $ L^1(a,b)$, $\gamma_{\varepsilon} \to \gamma$ in the flat norm, and
\begin{equation*}
\limsup_{\varepsilon \to 0} E_{\varepsilon}(u_{\varepsilon}, \gae)\leq E(u,\gamma). 
\end{equation*}
\end{enumerate}
\end{theorem}

The associated compactness result is stated in Theorem~\ref{com}, where we in fact establish for the second variable convergence in $\WMQ$ for all $1 < q < \infty$. Therefore, we obtain as a direct consequence of Theorem~\ref{mainresult} also $\Gamma$-convergence of $\{E_{\varepsilon}\}_{\varepsilon>0}$ to~$E$ in $L^1(a,b) \times \mathcal{M}(a,b)$, when $\mathcal{M}(a,b)$ is equipped with the stronger topology of convergence in $\WMQ$ for some $1 < q < \infty$.  
 
\begin{remark}
Our result can be seen as a two-field extension of the setting considered in~\cite{LuVi97}. Indeed, introducing the constraints $g = u'$, respectively $\gamma = Du$, one is lead to functionals $F_\varepsilon(u) = E_\varepsilon(u,u'\mathcal{L}^1)$ and $F(u) = E(u, Du)$ depending on $u$ only. In this case the $\Gamma$-convergence of the sequence $\{F_\varepsilon\}_\varepsilon$ to $F$ has been obtained in~\cite{LuVi97}. 
\end{remark}

The unrestricted problem is in fact strongly related. Indeed, even for $K = \infty$ an energy bound implies $L^\infty$ bounds away from an asymptotically small exceptional set. The complement of the exceptional set can be chosen as a union of a bounded number of intervals, concentrating on the points of a finite partition $a = x_0 < x_1 < \ldots < x_m = b$ of $(a, b)$ in the limit $\varepsilon \to 0$ such that $\{u_\varepsilon\}_\varepsilon$ and $\{g_{\varepsilon}\}_\varepsilon$ converge with respect to the $L^1$ norm, respectively the flat norm, locally on $(a, b) \setminus \{ x_1, \ldots, x_{m-1} \}$. On the exceptional set, however, $u_{\varepsilon}'$ and $g_{\varepsilon}$ can assume extremely large values, spoiling their convergence even in a weak distributional sense. As a result, large jumps can develop in the limit and parts of $u_\varepsilon$ may elapse to $\pm \infty$. In order to account for such a possibility we consider limiting functions taking values in $\overline{\mathbb{R}} = \mathbb{R} \cup \{-\infty,+\infty\}$. More precisely, let $\mathcal{P} = \big\{ (x_0, \ldots, x_m) \colon a = x_0 < x_1 < \ldots < x_{m} = b, \ m \in \mathbb{N} \big\}$ and consider $BV_{\infty,\mathcal{P}}(a,b)$ as consisting of functions $u  \colon (a, b) \to \overline{\R}$ of the form 
\begin{equation}\label{eq:u-w-alphai}
  u = w + \sum_{i=1}^m \alpha_i \chi_{(x_{i-1}, x_i)} 
\end{equation}
with $\alpha_i \in \{-\infty, 0, +\infty\}$, $i = 1, \ldots, m$, $(x_0, \ldots, x_m) \in \mathcal{P}$ and $w \in BV(a, b)$. We denote the part where $u$ is finite by $\mathcal{F}(u) = \big( \bigcup_{i: \alpha_i = 0}[x_{i-1},x_i] \big)^\circ$, set $J_u = \{x \in (a, b) \colon [u](x) = u(x+) - u(x-) \in \overline{\R} \setminus \{0\}\}$ and read $\ff{\infty} = c_0$. We then extend $E$ to $L^0((a,b), \overline{\R}) \times \M$ by setting 
\begin{equation*}
E(u,\gamma) \coloneqq
\begin{cases}
\int_{\mathcal{F}(u)} {|u^{\prime}- g|^2 \dd x }+ c_0 \int_{\mathcal{F}(u)} {  |g| \dd x } & \quad \text{if } u \in BV_{\infty,\mathcal{P}}(a,b), \,  g \in L^1(\mathcal{F}(u)), \\ 
\quad + 2 \, \sum_{x \in J_u} \ff{\frac{1}{2}|[u](x)|} & \quad \quad \gamma\mrs\mathcal{F}(u) = (D^su + g \mathcal{L}^1)\mrs\mathcal{F}(u),\\  \quad +c_0 |D^cu|(\mathcal{F}(u))
& \quad \phantom{if } \text{and }  u^{\prime}- g \in L^2(\mathcal{F}(u)),\\[0.2cm] 
\infty & \quad \text{otherwise}. 
\end{cases} 
\end{equation*}
We say that $(u_\varepsilon, \gamma_\varepsilon) \to (u, \gamma)$ in $L^0((a,b), \overline{\R}) \times \M$ if $u_\varepsilon \to u$ a.e.\ and $\gamma_\varepsilon \to \gamma$ in the flat norm locally in $\mathcal{F}(u)$ on the complement of a finite set, i.e., $\gamma_\varepsilon \mrs A \to \gamma \mrs A$ on each open set $A \Subset \mathcal{F}(u) \setminus \{x_1, \ldots, x_{m-1}\}$ for some $(x_0, \ldots, x_m) \in \mathcal{P}$. With this notion of convergence we have:  
\begin{theorem}\label{mainresultKinfty}
Assume $K=\infty$. Then the family $\{E_{\varepsilon}\}_{\varepsilon>0}$ $\Gamma$-converges to~$E$ in $L^0((a,b), \overline{\R}) \times \M$. 
\end{theorem}

The corresponding compactness result for $K=\infty$ with respect to this particular convergence is stated in Theorem~\ref{com-K-infty}. 

\begin{remark}
In fact, $E(u, \gamma)$ can be finite only if the restriction of $u$ to $\mathcal{F}(u)$ is a $BV$ function and not merely an element of $GBV(\mathcal{F}(u))$. In particular, if $E(u, \gamma) < \infty$ and $u \in L^1(a, b)$, then $u \in BV(a, b)$. 
\end{remark}

\begin{remark}
Our methods can easily be adapted to obtain an alternative asymptotic description by considering renormalized functionals: From the above partition $\mathcal{P}$ one can derive a coarser one (whose members are finite unions of intervals) so that on each such set one has an $L^\infty$ bound on $u_\varepsilon$ modulo a single additive constant and the mutual distance of $u_\varepsilon$ on two different sets diverges. This allows for an asymptotic description also of those parts that escape to infinity.  
\end{remark}

\begin{remark}
Our results remain true if restricted to preassigned boundary values $u(a) = u_a$, $u(b) = u_b$, where $u_a, u_b \in \R$ such that $-K \le u_a, u_b \le K$. As for a bounded energy sequence parts of the jump set could accumulate at the boundary $\{a, b\}$, a usual way to implement boundary conditions is to consider $E^{u_a,u_b}_\varepsilon$ and $E^{u_a,u_b}$ on an extended interval $(a - \eta, b + \eta)$, with $\eta > 0$ fixed, which are defined as $E_{\varepsilon}$ and~$E$, respectively, before but with the additional constraints $E^{u_a,u_b}_\varepsilon(u,\gamma) = E^{u_a,u_b}(u, \gamma) = \infty$ if not $u = u_a$ a.e.\ on $(a - \eta, a)$ and $u = u_b$ a.e.\ on $(b, b + \eta)$ and $\gamma \mrs ( (a - \eta, a) \cup (b, b + \eta) ) = 0$. So if $u$ does not satisfy the given boundary values on $(a, b)$ in the limiting problem, this leads to an extra energy cost: 
\[ E^{u_a,u_b} (u, \gamma) 
   = E \big( u|_{(a,b)}, \gamma \mrs (a,b) \big) 
     + 2 f \left( \frac{1}{2} |u(a+) - u_a| \right) + 2 f \left( \frac{1}{2} |u(b-) - u_b| \right). \] 
Indeed, the $\Gamma$-$\liminf$ inequality and the compactness property are direct consequences of the case with free boundaries. The $\Gamma$-$\limsup$ inequality for $K < \infty$ follows from the observation that the recovery sequence constructed in Proposition~\ref{SBVupperbpundestimate} indeed satisfies $u_{\varepsilon} = u$ near $\{a,b\}$ and from Remark~\ref{rmk:SBV-relax-bv}. The case $K = \infty$ is a direct consequence as there the recovery sequence is built as in the case $K < \infty$ near $\{a,b\}$ since $u_a, u_b \in \R$. 
\end{remark}

\begin{remark}\label{rmk:gen-stored-en}
Our results can also be adapted to general continuous stored energy functions~$W$ leading to a general non-quadratic bulk contribution $\int_a^b W(u' - g) \dd x$, whenever~$W$ satisfies a $p$-growth condition of the form $c|r|^p - C \le W(r) \le C|r|^p + C$ for suitable constants $c, C > 0$ and some $p \in (1, \infty)$, and for convenience we also assume that $\min W = W(0) = 0$. The first term in the limiting functional is then replaced by $\int_a^b W^{**}(u' - g) \dd x$, respectively, $\int_{\mathcal{F}(u)} W^{**}(u' - g) \dd x$, where~$W^{**}$ is the convex envelope of~$W$. In fact, making use of the estimate $W \ge W^{**}$, compactness and the $\Gamma$-$\liminf$ inequality follow exactly as before with $W^{**}$ instead of $|\cdot|^2$ by taking account of the obvious adaptions such as replacing $L^2$ by $L^p$ and $SBV^2$ by $SBV^p$. The $\Gamma$-$\limsup$ inequality requires an extra relaxation step, which is detailed in Remark~\ref{rmk:W-Wrelax}, and is otherwise straightforward as well. 
\end{remark}

For completeness we also give the corresponding approximation results for the minimal energies with respect to the second variable $\gamma$, which are defined as 
\begin{equation}
\label{def_minimal_energy_eps}
\tilde{E}_{\varepsilon}(u) \coloneqq  \inf_{\gamma \in \M} E_{\varepsilon}(u,\gamma) = \inf_{g \in L^1(a,b)} E_{\varepsilon}(u,g \mathcal{L}^1)
\end{equation}
and 
\begin{equation}
\label{def_minimal_energy_limit}
\tilde{E}(u) \coloneqq  \inf_{g \in L^1(a,b)} E(u,D^su+ g \mathcal{L}^1), 
\end{equation}
for $u \in L^1(a,b)$ and $u \in L^0((a,b), \overline{\R})$, respectively. As a direct consequence of the previous $\Gamma$-convergence result we obtain: 

\begin{corollary}\label{mainresult-K-infty}
The family $\{\tilde{E}_{\varepsilon}\}_{\varepsilon>0}$ $\Gamma$-converges to $\tilde{E}$, on $L^1(a,b)$ equipped with the strong topology if $K<\infty$ and with respect to convergence a.e.\ on $L^0((a,b), \overline{\R})$ if $K=\infty$. 
\end{corollary}

\section{Preliminaries} \label{sec_preliminaries}

In this section, we recall some basics on $BV$-functions, for simplicity on a one-dimensional interval $(a,b)\subset \R$, and convergence of measures.  

\paragraph{Functions of bounded variation.} A function $u\in L^1(a,b)$ is said to belong to the space $BV(a,b)$ of \emph{functions of bounded variation} if its distributional derivative
is a finite Radon measure, i.e, if the integration-by-parts formula
\begin{equation*}
\int_{a}^b u \varphi' \dd x = - \int_{a}^b \varphi \dd Du \quad \text{ for every } \varphi \in C_c^{1}(a,b)
\end{equation*}
is valid for a (unique) measure $Du \in \mathcal{M}(a,b)$. The space $BV(a,b)$ is a Banach space endowed with the norm
\begin{equation*}
\norm{u}_{BV(a,b)} \coloneqq  \norm{u}_{L^1(a,b)} + |Du|(a,b),
\end{equation*}
where $|Du|(a,b)$ is the total variation of~$Du$. We here collect some basic facts from~\cite{afp} for functions of bounded variation, which are relevant for our paper.  

We recall the notions of weak-$*$ and strict convergence for sequences $\{u_n\}_n$ in $BV(a,b)$, which are useful for compactness properties and approximation arguments, respectively. We say that $\{u_n\}_n$ \emph{converges weakly-$*$} to $u \in BV(a,b)$, denoted by $u_n \overset{*}{\rightharpoonup} u$, if $u_n \to u$ in $L^1(a,b)$ and $Du_n \overset{*}{\rightharpoonup} Du$ in $\mathcal{M}(a,b)$. We notice that every weakly-$*$ converging sequence in $BV(a,b)$ is norm-bounded by Banach--Steinhaus, while every norm-bounded sequence in $BV(a,b)$ contains a weakly-$*$ converging subsequence (see \cite[Theorem~3.23]{afp}). We further say that $\{u_n\}_n$ \emph{converges strictly} to $u \in BV(a,b)$ if  $u_n \to u$ in $L^1(a,b)$ and $|Du_n|(a,b) \to |Du|(a,b)$. As a matter of fact, the space $C^{\infty}(a,b)$ is dense in $BV(a,b)$ with respect to the strict topology (see \cite[Theorem~3.9]{afp}). 

We next discuss approximate continuity and discontinuity properties of a function $u \in L^1_{\textnormal{loc}}(a,b)$. We say that~$u$ has an \emph{approximate limit} at $x \in (a,b)$ if there exists a (unique) $\widetilde{u}(x) \in \R$ such that
\begin{equation*}
\lim_{\rho \to 0^+} \dashint_{(x-\rho,x+\rho) \cap (a,b)} |u(y)-\widetilde{u}(x)| \dd y =0.  
\end{equation*}
We denote by~$S_u$ the set, where this condition fails, and call it the \emph{approximate discontinuity set} of~$u$. It is $\mathcal{L}^1$-negligible, and $\widetilde{u}$ coincides $\mathcal{L}^1$-a.e.~in $(a,b) \setminus S_u$ with~$u$.  Furthermore, we say that~$u$ has an \emph{approximate jump point} at $x \in S_u$ if there exist (unique) $u(x+), u(x-) \in \R$ with $u(x+) \neq u(x-)$ such that 
\begin{equation*}
\lim_{\rho \to 0^+} \dashint_{(x,x+\rho)} |u(y)-u(x+)| \dd y =0 \quad \text{and} \quad \lim_{\rho \to 0^+} \dashint_{(x-\rho,x)} |u(y)-u(x-)| \dd y =0.
\end{equation*}
We denote by~$J_u$ the set of approximate jump points and call it the \emph{jump set} of~$u$. Notice that $u(x+)$ and $u(x-)$ can be considered as one-sided limits from the right and from the left, respectively.

For $u \in BV(a,b)$ the set~$J_u$ coincides with~$S_u$ and is at most countable. It is also convenient to work with the precise representative
\begin{equation*}
u^\ast(x)  \coloneqq 
\begin{cases}
\widetilde{u}(x) & \quad \text{if } x \in (a,b)\setminus J_u,\\
[u](x)/2 & \quad \text{if } x\in J_u.
\end{cases}
\end{equation*}
According to the Radon--Nikod\'ym theorem the measure derivative~$Du = D^au \mathcal{L}^1 + D^su$ can be decomposed into the absolutely continuous and the singular part with respect to the Lebesgue measure~$\mathcal{L}^1$. We then define the jump and the Cantor part of~$Du$ as
\begin{equation*}
D^j u  \coloneqq  D^s u \mrs J_u \quad \text{and} \quad  D^c u  \coloneqq   D^s u \mrs ((a,b)\setminus J_u).
\end{equation*}
From the identifications $D^au =  u' \mathcal{L}^1$ with the approximate gradient~$u'$ for the absolutely continuous part and $D^ju = [u] \mathcal{H}^0  \mrs J_u$ with $[u] \coloneqq u(\cdot \, +) - u(\cdot \, -)$ for the jump part (see \cite[Theorem~3.83 and formula (3.90)]{afp}) we arrive at the decomposition
\begin{equation*}
Du= u' \mathcal{L}^1 + [u] \mathcal{H}^0  \mrs J_u + D^c u.
\end{equation*} 
We can actually decompose the function~$u$ as  
\begin{equation}
\label{decomposition_u_1d}
 u = u_a+u_j+u_c
\end{equation}
for an absolutely continuous function~$u_a \in W^{1,1}(a,b)$ with $Du_a = D^a u$, a jump function~$u_j$ with $Du_j=D^ju$, and a (continuous) Cantor function $u_c$ with $Du_c=D^cu$ (notice that these functions are determined uniquely up to additive constants). Thus, the decomposition of~$Du$ is recovered from a corresponding decomposition of the function itself (which for $BV$-functions defined on open subsets of~$\R^d$ with $d>1$ in general is not possible). 

We finally mention the subspace $SBV(a,b)$ of \emph{special functions of bounded variation}, which contains all functions $u \in BV(a,b)$ with $D^cu= 0$. In addition, we define
\begin{equation*}
SBV^2(a,b) \coloneqq  \big\lbrace u \in SBV(a,b) \colon u' \in L^2(a,b) \text{ and } \mathcal{H}^0 (J_u) < \infty \big\rbrace.
\end{equation*}

\paragraph{Convergence in negative Sobolev spaces and in the flat topology.} The negative Sobolev spaces $W^{-1,q}(a,b)$ with $1 < q \leq \infty$ are defined as usual as the dual spaces of $W^{1,q'}_0(a,b)$ (with $q' \in [1,\infty)$ denoting the conjugate exponent to~$q$ with $\tfrac{1}{q} + \tfrac{1}{q'}=1$), and correspondingly the norm is defined via the duality pairing as 
\begin{equation*}
 \| T \|_{W^{-1,q}(a,b)}  \coloneqq  \sup \Big \{ T(\varphi) \colon \varphi \in W^{1,q'}_0(a,b) \text{ with }\norm{\varphi}_{W^{1,q'}(a,b)} \leq 1 \Big\},
\end{equation*}
for every $T \in W^{-1,q}(a,b)$. Consequently, the spaces $W^{-1,q}(a,b)$ with $1 < q < \infty$ are reflexive and separable. For later purposes, we mention two specific situations. Let $v \in L^q(a,b)$ and $w \in L^r(a,b)$ with~$1 \leq r \leq \infty$ such that the embedding $W^{1,q'}_0(a,b) \subset L^{r'}(a,b)$ is continuous. If we set
\begin{equation*}
 T_{Dv}(\varphi)  \coloneqq  - \int_a^b v \varphi' \diff x \quad \text{and} \quad  T_w(\varphi)  \coloneqq  \int_a^b w \varphi \diff x \quad \text{for all } \varphi \in W^{1,q'}_0(a,b), 
\end{equation*}
then, by the H{\"o}lder inequality and the continuous embedding (with constant~$C'$), we obtain $T_{Dv}, T_w \in W^{-1,q}(a,b)$ with
\begin{equation}
\label{eq:neg_Sobolev_trivial_estimate}
 \| T_{Dv} \|_{W^{-1,q}(a,b)} \leq \norm{v}_{L^q(a,b)} \quad \text{and} \quad \| T_w \|_{W^{-1,q}(a,b)} \leq C' \norm{w}_{L^r(a,b)}. 
\end{equation}
Because of the continuous and dense embedding $W^{1,1}_0(a,b) \subset C_0(a,b)$, the negative Sobolev norms can actually be considered on the space~$\M$ of all finite Radon measures on~$(a,b)$, for which the duality pairing reads as 
\begin{equation*}
 \| \mu \|_{W^{-1,q}(a,b)} = \sup \bigg\{ \int_a^b \varphi \dd \mu \colon \varphi \in W^{1,q'}_0(a,b) \text{ with } \norm{\varphi}_{W^{1, q'}_0(a,b)}\leq 1 \bigg\}
\end{equation*}
for $\mu \in \M$. If we allow $q' = \infty$ in this expression, we obtain the flat norm
\begin{equation*}
\normf{\mu} \coloneqq  \sup \bigg\{ \int_a^b \varphi \dd \mu \colon \varphi \in W^{1, \infty}_0(a,b) \text{ with } \norm{\varphi}_{W^{1, \infty}_0(a,b)}\leq 1 \bigg\},
\end{equation*}
for $\mu \in \M$. Here we have the inequalities 
\begin{equation}
\label{eq:flat_trivial_estimate}
 \normf{Dv} \leq  \norm{v}_{L^1(a,b)}  \quad \text{and} \quad  \|w \mathcal{L}^1\|_{\text{flat}} \leq \norm{w}_{L^1(a,b)} 
\end{equation}
for all functions $v \in BV(a,b)$ and $w \in L^1(a,b)$. Let us still notice that due to Schauder's theorem and the compact embedding $W^{1,\infty}_0(a,b) \Subset C_0(a,b)$, the flat topology metrizes weak-$\ast$ convergence of (uniformly bounded) measures. Therefore, we have the following relations for the convergence of measure with respect to convergence in $W^{-1,q}(a,b)$, the flat norm and in the weak-$\ast$ sense. 

\begin{lemma}[on convergence of measures] 
\label{Lemma_weak_negativ_flat}
For a measure $\mu \in \M$ and a sequence $\{\mu_n\}_{n \in \N}$ of measures in $\M$, we have: 
\begin{enumerate}[font=\normalfont, label=(\roman{*}), ref=(\roman{*})]
 \item If $\mu_n \to \mu$ in $W^{-1,q}(a,b)$ for some  $1 < q \leq \infty$, then $\mu_n \to \mu$ in the flat norm.
 \item If $\mu_n \overset{*}{\rightharpoonup} \mu$ in $\M$, then $\mu_n \to \mu$ in the flat norm.
\item If $\mu_n \to \mu$ in the flat norm and $\sup_{n \in \N} |\mu_n|(a,b) < \infty$, then $\mu_n \overset{*}{\rightharpoonup} \mu$ in $\M$.
\end{enumerate}
\end{lemma}

\section{Compactness}\label{sec: compactness1d}
In this section we establish a compactness result for sequences in $L^1(a,b) \times \mathcal{M}(a,b)$ with bounded energy~$E_\varepsilon$. This result together with a ${\Gamma\text{-convergence}}$ result implies the convergence of minimizers and the corresponding minimum values. In order to bound suitable norms of the two fields in terms of the energy, we first prove the following technical lemma:

\begin{lemma}\label{app}
Let $g \in L^1(a,b)$. For $\varepsilon >0$ there exists $x_\varepsilon \in \R$ such that the grid   
\begin{equation*}
\Ge  \coloneqq  \big\{ x_\alpha \coloneqq x_\varepsilon + 2 \varepsilon \alpha \colon \alpha \in \Z \text{ and } x_\alpha \in (a + \varepsilon, b - \varepsilon) \big\}
\end{equation*}
contains a subset $\Gee \subset \Ge$ with 
\begin{equation*}
\int_{\bigcup \{I_\varepsilon(x_\alpha) \colon x_\alpha \in \Gee\}} |g|\diff t + 2\# (\Ge \setminus \Gee) \leq \frac{1}{c_0 \varepsilon} \int_a^b f \bigg( \varepsilon \dashint_{I_\varepsilon(x)\cap (a,b)}{ |g|\diff t } \bigg) \dd x .
\end{equation*}
\end{lemma}

\begin{proof}
We proceed analogously to \cite[proof of Lemma 4.2]{LuVi97}.  Let $\phi_{\varepsilon} \in C^{\infty}_0(a, b)$ be a cut-off function with $0 \leq \phi_{\varepsilon} \leq 1$ in $(a,b)$ and $\phi_{\varepsilon} \equiv 1$ in $(a + \varepsilon,b - \varepsilon)$. We then consider $\psi_{\varepsilon} \in C_0(\R)$ defined via 
\begin{equation*}
\psi_{\varepsilon}(x) \coloneqq  \phi_{\varepsilon}(x)\, f \bigg( \varepsilon \dashint_{I_\varepsilon(x)\cap (a,b)}{ |g|\diff t } \bigg)
\end{equation*}
for $x \in (a,b)$ and  $\psi_{\varepsilon}(x) \coloneqq 0$ for $x \in \R \setminus (a,b)$. The application of \cite[Lemma 4.2]{BDM97} (with $\eta = 2 \varepsilon$), which is a consequence of the mean value theorem for integrals, shows that 
\begin{equation*}
 \int_{\R} \psi_\varepsilon \dd x= 2 \varepsilon \sum_{\alpha \in \Z} \psi_\varepsilon (x_\varepsilon+ 2 \varepsilon \alpha)
\end{equation*}
holds for a suitable $x_\varepsilon \in \R$. By non-negativity of~$f$, the choice of the cut-off function~$\phi_\varepsilon$ and the definition of $\Ge$, we hence have
\begin{align}
 \frac{1}{\varepsilon} \int_a^b f \bigg( \varepsilon \dashint_{I_\varepsilon(x)\cap (a,b)}{ |g|\diff t } \bigg) \dd x 
 & \geq  \frac{1}{\varepsilon} \int_\R \psi_{\varepsilon}(x) \diff x \nonumber \\
 & \geq 2 \sum_{x_\alpha \in \Ge} \psi_{\varepsilon}(x_{\alpha}) 
 =2\sum_{x_\alpha \in \Ge} f \bigg( \varepsilon  \dashint_{I_\varepsilon(x_{\alpha}) }{ |g| \diff t }\bigg). \label{eq_f_integral_mv}
\end{align}
If we now set 
\begin{equation*}
\Gee \coloneqq  \bigg\{ x_\alpha \in \Ge \colon  \varepsilon \dashint_{I_\varepsilon(x_{\alpha})} |g|\diff t < 1 \bigg\},
\end{equation*}
then the claim follows directly from~\eqref{eq_f_integral_mv}, after rewriting the right-hand side via the definition of~$f$ as
\begin{align*}
 2\sum_{x_\alpha \in \Ge} f \bigg( \varepsilon  \dashint_{I_\varepsilon(x_{\alpha}) }{ |g| \diff t }\bigg) 
 & = 2 c_0 \sum_{x_\alpha \in \Gee} \varepsilon  \dashint_{I_\varepsilon(x_{\alpha}) }{ |g| \diff t } +  2\sum_{x_\alpha \in \Ge \setminus \Gee} c_0 \\
 & = c_0  \sum_{x_\alpha \in \Gee} \int_{I_\varepsilon(x_{\alpha}) }  |g| \diff t + 2c_0 \# (\Ge \setminus \Gee). \qedhere
\end{align*}
\end{proof}

We can now address the aforementioned compactness results.

\begin{theorem}[Compactness for $K < \infty$]\label{com}
Assume $K<\infty$. Let $\{(u_{\varepsilon},\gamma_\varepsilon)\}_{\varepsilon}$ be a sequence in $\LE \times \M$ with
\begin{equation*}
 E_{\varepsilon}(\ue,\gae)\leq C_0 
\end{equation*}
for a positive constant~$C_0$ and all $\varepsilon>0$. There exist a function $u \in BV(a,b)$ with $\norm{u}_{\LU} \leq K$ and a measure $\gamma \in \M$ such that, up to subsequences,  $\{u_{\varepsilon}\}_{\varepsilon}$ converges to $u$ in $L^1(a,b)$ and $\{\gamma_{\varepsilon}\}_{\varepsilon}$ converges to $\gamma=\gamma^s+ g \mathcal{L}^1$ in $\WMQ$ for all $1 < q < \infty$ and in particular in the flat norm. Moreover, there holds $\gamma^s = D^su$ and $u^{\prime} - g \in L^2(a,b)$.
\end{theorem}

\begin{proof}
We first observe from the finiteness of $E_{\varepsilon}(u_\varepsilon,\gamma_\varepsilon)$ that we necessarily have $\ue \in \W$ with $\norm{\ue}_{L^{\infty}(a,b)}\leq K$ and $\gae = g_{\varepsilon}\mathcal{L}^1$ for some functions $\ggae \in L^1(a,b)$ for all~$\varepsilon>0$.  By the uniform boundedness of $\norm{\ue^{\prime}- g_{\varepsilon}}_{L^2(a,b)}$ and $\norm{\ue}_{L^{\infty}(a,b)}$ we deduce from~\eqref{eq:neg_Sobolev_trivial_estimate}
\begin{align*}
\norm{\gae}_{W^{-1, \infty}(a,b)} &\leq \norm{\ue^{\prime} \mathcal{L}^1}_{W^{-1, \infty}(a,b)} + \norm{\gae - \ue^{\prime}\mathcal{L}^1}_{W^{-1, \infty}(a,b)}\notag\\ 
&\leq \norm{\ue}_{L^{\infty}(a,b)} +  C^{\prime} \norm{g_{\varepsilon} - \ue^{\prime}}_{L^2(a,b)} \leq C(a,b,K), 
\end{align*}
independently of~$\varepsilon$. Therefore, $\{\gae\}_{\varepsilon}$ is bounded in $W^{-1, \infty}(a,b)$ and consequently contains a subsequence, which converges weakly-$\ast$ in $W^{-1, \infty}(a,b)$ to some $\gamma \in W^{-1, \infty}(a,b)$. We next study the convergence of the sequence $\{u_{\varepsilon}\}_{\varepsilon}$. To this end, we consider the function~$v_{\varepsilon}$ defined by 
\begin{equation*}
v_{\varepsilon}(x) \coloneqq  \begin{cases}
\ue(x) & \quad  x \in \bigcup \{I_\varepsilon(x_\alpha) \colon x_\alpha \in \Gee\},  \\ 
0 & \quad \text{otherwise}.
\end{cases} 
\end{equation*}
Since $\ue \in W^{1,1}(a,b)$ with $\norm {u_{\varepsilon}} _{L^{\infty}(a,b)} \leq K$ is assumed, we clearly have ${v_{\varepsilon} \in SBV(a,b)}$ with $\norm {v_{\varepsilon}} _{L^{\infty}(a,b)} \leq K$, for every~$\varepsilon$. Moreover, jump discontinuities of $v_\varepsilon$ can only occur at points $x_\alpha \pm \varepsilon$ for $x_\alpha \in \Ge \setminus \Gee$ and close to the boundary at $\min \Gee$ or at $\max  \Gee$. As a consequence of Lemma~\ref{app} and the definition of the energy~$E_\varepsilon$, we deduce 
\begin{equation*}
\# J_{v_{\varepsilon}} \leq  2\# (\Ge \setminus \Gee) + 2 
 \leq \frac{1}{c_0 \varepsilon} \int_a^b f \bigg( \varepsilon \dashint_{I_\varepsilon(x)\cap (a,b)}{ |g|\diff t } \bigg) \dd x + 2 \leq \frac{C_0}{c_0} +2,
\end{equation*}
i.e., that $\# J_{v_{\varepsilon}}$ is bounded independently of $\varepsilon$. By the Cauchy--Schwarz inequality and again by Lemma~\ref{app}, we additionally have
\begin{align*}
\int_a^b{ |v^{\prime}_{\varepsilon}(x)|\diff x } 
&\leq \int_{ \bigcup \{I_\varepsilon(x_\alpha) \colon x_\alpha \in \Gee\}} |u^{\prime}_{\varepsilon}(x)- g_{\varepsilon}(x)| \diff x + \int_{ \bigcup \{I_\varepsilon(x_\alpha) \colon x_\alpha \in \Gee\}} |g_{\varepsilon}(x)| \diff x\\
&\leq  |b-a|^{\frac{1}{2}} \left( \int_a^b |u^{\prime}_{\varepsilon}(x)- g_{\varepsilon}(x)|^2 \diff x \right)^{\frac{1}{2}}+  \frac{1}{c_0 \varepsilon} \int_a^b f  \bigg( \varepsilon \dashint_{I_\varepsilon(x)\cap (a,b)}{ | g_{\varepsilon}|\diff t } \bigg) \dd x
\\ 
& \leq |b-a|^{\frac{1}{2}} C_0^{\frac{1}{2}} +  \frac{C_0}{c_0},
\end{align*}
independently of~$\varepsilon$. Hence, $\{v_{\varepsilon}\}_{\varepsilon}$ is bounded in $BV(a,b)$. By the Rellich--Kondrachov theorem, there exists a subsequence that converges in $L^q(a,b)$ to some $u \in BV(a,b)$ with $\norm{u}_{\LU} \leq K$, for any $1 \leq q < \infty$. To identify~$u$ as the limit in $L^q(a,b)$ of (the same subsequence of) $\{u_{\varepsilon}\}_{\varepsilon}$ we notice from the definition of $v_\varepsilon$ 
\begin{equation*}
\norm{v_{\varepsilon}-\ue}^q_{L^q(a,b)}  \leq \int_{(a,b) \setminus  \bigcup \{I_\varepsilon(x_\alpha) \colon x_\alpha \in \Gee\}} K^q \dd x \leq 2 \varepsilon (\# (\Ge \setminus \Gee)+2) K^q.
\end{equation*}
Since $\# (\Ge \setminus \Gee)$ is bounded uniformly in~$\varepsilon$, this allows to conclude the convergence of $\{\ue\}_{\varepsilon}$ to~$u$ in $L^q(a,b)$, for any $1 \leq q < \infty$.

It remains to show convergence of (the subsequence of) $\{\gae\}_{\varepsilon}$ in $W^{-1,q}(a,b)$ for any $1 < q < \infty$ and the claimed relations between~$\gamma$ and~$Du$. In view of~\eqref{eq:neg_Sobolev_trivial_estimate}, we have $u^{\prime}_{\varepsilon} \mathcal{L}^1 \to Du=u^{\prime} \mathcal{L}^1 + D^su$ in $W^{-1,q}(a,b)$ for every $1 < q < \infty$.
Furthermore, by the boundedness of $\{u^{\prime}_{\varepsilon}- g_{\varepsilon}\}_{\varepsilon}$ in $L^2(a,b)$, we can extract a subsequence which converges weakly in $L^2(a,b)$ to some $w \in L^2(a,b)$. Since $L^2(a,b)$ is compactly embedded in $W^{-1,q}(a,b)$ for all ${1 < q < \infty}$, we obtain $\gae \to Du-w\mathcal{L}^1$ in $W^{-1,q}(a,b)$ and thus, via Lemma~\ref{Lemma_weak_negativ_flat}, also in the flat norm. This shows $ \gamma=Du-w\mathcal{L}^1 \in \M$, which in turn, by the Radon--Nikodým theorem, yields $\gamma^s= D^su$ and $w= u^{\prime} - g \in L^2(a,b)$.
\end{proof}

\begin{theorem}[Compactness for $K = \infty$]\label{com-K-infty}
Assume $K=\infty$. Suppose $\{(u_{\varepsilon},\gamma_\varepsilon)\}_{\varepsilon}$ is a sequence in $L^0((a, b); \overline{\R}) \times \M$ with
\begin{equation*}
 E_{\varepsilon}(\ue,\gae)\leq C_0 
\end{equation*}
for a positive constant~$C_0$ and all $\varepsilon>0$. There is a partition $a = x_0 < x_1 < \ldots < x_m = b$, a function $u \in BV_{\infty,\mathcal{P}}(a,b)$, say 
\[ u = w + \sum_{i=1}^m \alpha_i \chi_{(x_{i-1}, x_i)} \] 
with $\alpha_i \in \{-\infty, 0, +\infty\}$, $i = 1, \ldots, m$ and $w \in BV((a, b); \mathbb{R})$, and a measure $\gamma \in \M$ such that, up to subsequences, $\{u_{\varepsilon}\}_{\varepsilon}$ converges to $u$ a.e.\ and in $L^1_{\rm loc}(\mathcal{F}(u) \setminus \{x_1, \ldots, x_m\})$ and $\{g_{\varepsilon}\mathcal{L}^1\}_{\varepsilon}$ converges to $\gamma=D^s w+ g \mathcal{L}^1$ in $W^{-1,q}_{\rm loc}((a,b) \setminus \{x_1, \ldots, x_m\})$ for all $1 < q < \infty$ and in particular in the flat norm on each open $I$ that is compactly contained in $(a,b) \setminus \{x_1, \ldots, x_m\}$. Moreover, $w^{\prime} - g \in L^2(a,b)$. In particular, $(u_\varepsilon, \gamma_\varepsilon) \to (u,\gamma)$ in $L^0((a, b); \overline{\R}) \times \M$. 
\end{theorem}

\begin{proof}
We define $\mathcal{G}'_{\varepsilon}$ exactly as in Lemma~\ref{app} and denote by $J_1, \ldots, J_{m_{\varepsilon}}$ the connected components of $\bigcup \{\overline{I_{\varepsilon}(x_{\alpha})} : x_{\alpha} \in \mathcal{G}'_{\varepsilon}\}$. We then notice that the arguments in the preceding proof show that for some constant~$C$ we have $m_{\varepsilon} \le C$ and 
\begin{align}\label{eq:J-int-bds} 
  \int_{J_1 \cup \ldots \cup J_{m_{\varepsilon}}} | u_\varepsilon' (x)| \, \mathrm dx \le C, 
\end{align}
while $\mathcal{L}^1((a, b) \setminus (J_1 \cup \ldots \cup J_{m_{\varepsilon}})) \le C \varepsilon$. 

We set $x_{\varepsilon,i} = \sup J_i$ and $\alpha_{\varepsilon,i} = \dashint_{J_i} u_{\varepsilon} \, \mathrm dx$, $i = 1, \ldots, m_{\varepsilon}$. Note that \eqref{eq:J-int-bds} implies 
\begin{align}\label{eq:pw-infty-bd} 
  \| u_\varepsilon - \alpha_{\varepsilon,i} \|_{L^\infty(J_i)} \le C. 
\end{align} 
Passing to a subsequence we may assume that $m_{\varepsilon} = \tilde{m}$ for some $\tilde{m}$ independent of $\varepsilon$ and $x_{\varepsilon,i} \to \tilde{x}_i \in [a,b]$ as well as $\alpha_{\varepsilon,i} \to \tilde{\alpha}_{i} \in \overline{\mathbb{R}} = \mathbb{R} \cup \{-\infty,+\infty\}$, for $i = 1, \ldots, \tilde{m}$. Choose $a = x_0 < x_1 < \ldots < x_{m} = b$ ($m \le \tilde{m}$) such that $\{x_0, \ldots, x_m\} = \{\tilde{x}_0, \tilde{x}_1, \ldots, \tilde{x}_{\tilde{m}}\}$ (with $\tilde{x}_0 := a$ and, by construction, $\tilde{x}_{\tilde{m}} = b$) and set $\alpha_i = \tilde{\alpha}_j$ if $\tilde{x}_{j-1} < \tilde{x}_j = x_i$. 

If $|\alpha_i| <  \infty$, then $\{u_\varepsilon\}_\varepsilon$ converges to a function $w$ weakly* in $BV_{\rm loc}(x_{i-1}, x_i)$ and the uniform bounds in \eqref{eq:J-int-bds} and \eqref{eq:pw-infty-bd} imply $\| w \|_{L^{\infty}(x_{i-1}, x_i)} + Dw((x_{i-1}, x_i)) \le C$, in particular, $w \in BV(x_{i-1}, x_i)$. If $\alpha_i = \pm \infty$, then $u_\varepsilon \to \pm \infty$ a.e.\ on $(x_{i-1}, x_i)$ by \eqref{eq:pw-infty-bd}. In this case we define $w \in BV(x_{i-1}, x_i)$ as the weak* limit in $BV_{\rm loc}(x_{i-1}, x_i)$ of $\{u_\varepsilon - \alpha_{i,\varepsilon}\}_\varepsilon$. The convergence of $\{\gamma_\varepsilon\}_\varepsilon$ follows directly by applying Theorem~\ref{com} on compact intervals $I$ of $(x_{i-1}, x_i)$ to $\{(u_\varepsilon, \gamma_\varepsilon)\}_\varepsilon$ and $\{(u_\varepsilon - \alpha_{i,\varepsilon}, \gamma_\varepsilon)\}_\varepsilon$, respectively. As $\| w' - g \|_{L^2(I)}$ is bounded independently of $I$ by the uniform energy bound, we have indeed $w' - g \in L^2(a,b)$. 
\end{proof} 


\begin{remark}\label{rmk:en-lb}
For later use we notice that Lemma~\ref{app} and the proof of Theorem~\ref{com-K-infty} show that, under the assumptions of Theorem~\ref{com-K-infty}, for any open $A \Subset (a, b) \setminus \{x_1, \ldots, x_{m-1}\}$ 
\begin{align*} 
  &\int_a^b { |u_\varepsilon' - g_\varepsilon |^2 \dd x } 
  + \frac{1}{\varepsilon} \int_a^b {f \bigg( \varepsilon \dashint_{I_{\varepsilon}(x) \cap (a,b)}{ |g_\varepsilon| } \dd t} \bigg) \dd x \\ 
  &\quad \ge 2 c_0 (m-1) + \int_A { |u_\varepsilon' - g_\varepsilon |^2 \dd x } 
   + \frac{1}{\varepsilon} \int_A {f \bigg( \varepsilon \dashint_{I_{\varepsilon}(x) \cap (a,b)}{ |g_\varepsilon| } \dd t} \bigg) \dd x 
\end{align*} 
for sufficiently small $\varepsilon$ since $\# (\Ge \setminus \Gee) \ge \tilde{m} - 1 \ge m - 1$.  
\end{remark}

\begin{remark}\label{L2assumptions}
According to the compactness results of Theorems~\ref{com} and~\ref{com-K-infty} we may restrict ourselves to pairs $(u, \gamma) \in BV(a,b) \times \M$ with $\norm{u}_{L^{\infty}(a,b)} \leq K$ and $(u, \gamma) \in BV_{\infty,\mathcal{P}}(a,b) \times \M$, respectively, and such that the measure $Du - \gamma$, is absolutely continuous with respect to the Lebesgue measure with density in $L^2(\mathcal{F}(u))$. The statements of Theorems~\ref{mainresult} and~\ref{mainresultKinfty} are trivial otherwise.
\end{remark}

\begin{remark}
In principle, with the compactness result of Theorem~\ref{com} at hand, we could infer our $\Gamma$-convergence result in Theorem~\ref{mainresult} from known results, by separate considerations of the elastic and inelastic contributions. To this end, given $(u_\varepsilon,\gamma_\varepsilon) \in W^{1,1}(a,b) \times \M$ with $\gamma_\varepsilon = g_\varepsilon \mathcal{L}^1$ and $u_\varepsilon' - g_\varepsilon \in L^2(a,b)$, we define $(W_\varepsilon,G_\varepsilon) \in W^{1,2}(a,b) \times W^{1,1}(a,b)$ via $W_\varepsilon(a) = G_\varepsilon(a) = 0$ and $W_\varepsilon'=u_\varepsilon-g_\varepsilon$,  $G_\varepsilon'=g_\varepsilon$ on $(a,b)$. This allows to express
\begin{equation*}
 E_\varepsilon(u_\varepsilon,\gamma_\varepsilon) = \int_a^b |W_\varepsilon'|^2  \dd x + \frac{1}{\varepsilon} \int_a^b f \bigg( \varepsilon \dashint_{I_\varepsilon(x)\cap (a,b)}{ |G_\varepsilon'|\diff t } \bigg) \dd x
\end{equation*}
as the sum of the standard Dirichlet energy for~$W_\varepsilon$ and a non-local energy involving only~$G_\varepsilon$ considered by Lussardi and Vitali in~\cite{LuVi97}. However, the understanding of the coupling between~$u_\varepsilon$ and~$\gamma_\varepsilon$ is essential for the extension to higher dimensions addressed in~\cite{AuerBeckSchmidt:21b}, which cannot be traced back to the one-dimensional case via the slicing technique. For this reason, we prefer to give a self-contained proof of Theorem~\ref{mainresult}.
\end{remark}

\section[{Estimate from below of the \texorpdfstring{$\Gamma$}{Gamma}-lower limit}]{Estimate from below of the \boldmath \texorpdfstring{$\Gamma$}{Gamma}-lower limit}
\label{sec: estimate from below  1d}

Next, we start with the proof of the $\Gamma$-$\liminf$ inequality. Except for the very last paragraph we assume $K < \infty$ in the whole section. To show this inequality it is useful to introduce the localized version of the functionals~$\{E_{\varepsilon}\}_{\varepsilon}$ and~$E$. They are defined for $(u,\gamma) \in L^1(a,b) \times \M$ and every open subset $A$ of $(a,b)$ via
\begin{align*}  
E_{\varepsilon}(u,\gamma,A)  \coloneqq  
\begin{cases}
\int_A { |u'- g |^2 \dd x }& \quad \text{if } u \in \W, \,  \norm{u}_{\LU} \leq K, \\ + \frac{1}{\varepsilon} \int_A {f \big( \varepsilon \dashint_{I_{\varepsilon}(x) \cap (a,b)}{ |g| } \dd t} \big) \dd x  & \quad \phantom{if } \gamma= g\mathcal{L}^1, \, g \in L^1(a,b),\\
 & \quad \phantom{if } \text{and }   u'-g \in  L^2(a,b),\\[0.2cm]
\infty & \quad \text{otherwise},
\end{cases}
\end{align*}
and 
\begin{equation*}
E(u,\gamma,A) \coloneqq 
\begin{cases}
\int_A {|u^{\prime}- g|^2 \dd x }+ \int_A { c_0 |g| \dd x } & \quad \text{if } u \in BV(a,b), \, \norm{u}_{\LU} \leq K, \\ 
+ 2 \, \sum_{x \in J_u \cap A} \ff{\frac{1}{2}|[u](x)|}& \quad \phantom{if }\gamma = D^su + g \mathcal{L}^1, \, g \in L^1(a,b),\\ 
+ c_0 |D^cu|(A) & \quad \phantom{if } \text{and } u^{\prime}- g \in L^2(a,b),\\[0.2cm] 
\infty & \quad \text{otherwise}. 
\end{cases} 
\end{equation*}
We denote the $\Gamma$-lower and $\Gamma$-upper limit of $\{E_{\varepsilon}\}_{\varepsilon}$ by
\begin{align*}
E^{\prime}(u, \gamma) & \coloneqq  \inf \big\{ \liminf_{\varepsilon \to 0} E_{\varepsilon}(\ue,\gae) \colon \ue \to u \text{ in } L^1(a,b), \, \gae \to \gamma \text{ in the flat norm}\big\}, \\
 E^{\prime \prime}(u,\gamma) & \coloneqq  \inf \big\{ \limsup_{\varepsilon \to 0} E_{\varepsilon}(\ue,\gae) \colon \ue \to u \text{ in } L^1(a,b), \, \gae \to \gamma \text{ in the flat norm}\big\},
\end{align*}
respectively. For the $\Gamma$-lower limit we also need a localized version, for which we adapt the notation and write $E^\prime(\cdot,\cdot,A)$ for every open subset~$A$ of $(a,b)$.

\begin{remark}[Properties of the localized $\Gamma$-lower limit]
\label{remark_properties_lower_limit}
Let $(u,\gamma) \in \LE \times \M$. 
\begin{enumerate}[font=\normalfont, label=(\roman{*}), ref=(\roman{*})]
 \item The properties that the set functions ${A \mapsto E_{\varepsilon}(u,\gamma,A)}$ are increasing and superadditive (on disjoint sets) for each $\varepsilon$ carry immediately over to ${A \mapsto E^{\prime}(u,\gamma,A)}$.
 \item A direct consequence of~(i) is that lower bounds for~$A \mapsto E'(u,\gamma,A)$ transfer from intervals in $(a,b)$ to arbitrary open subsets of~$(a,b)$, i.e., if for a positive Borel measure~$\lambda$ an estimate of the form
 \begin{equation*}
  E^{\prime}(u,\gamma, A)\geq \lambda(A) 
 \end{equation*}
 holds for all intervals $A \subset (a,b)$, then the estimate actually holds for any open subset $A \subset (a,b)$,  cp.~\cite[Remark~4.6]{LuVi97} for a similar statement. 
\end{enumerate}
\end{remark}


In this section we prove the $\Gamma$-$\liminf$ inequality, where in view of Remark~\ref{L2assumptions} it is sufficient to consider $(u, \gamma) \in  BV(a,b) \times \M$ with $\norm{u}_{L^{\infty}(a,b)} \leq K$,  $\gamma = D^su+ g \mathcal{L}^1$ and $u^{\prime} - g \in L^2(a,b)$. The basic idea is to derive three separate estimates for the jump part, the volume term and the Cantor term, respectively, and to infer the desired estimate then from a combination of these estimates by means of measure theory.

\subsection{Estimate from below of the jump term} \label{subsec: estimate from below jump part 1d}

\begin{Prop}\label{JPEP}
Let $A$ be an open subset of $(a,b)$. For every $(u,\gamma) \in  BV(a,b) \times \mathcal{M}(a,b)$  with $\gamma = D^su+ g \mathcal{L}^1$ and $u^{\prime}- g \in L^2(a,b)$ we have
\begin{equation*}
E^{\prime}(u,\gamma,A) \geq 2 \,\sum_{x \in J_u \cap A}  f \bigg( \frac{1}{2} |[u](x)|\bigg). 
\end{equation*}
\end{Prop}

\begin{proof}
\emph{Step 1: For every ${\bar{x}} \in J_u \cap A$ we have}
\begin{equation}
\label{eqn_jump_point_lower_estimate}
E^{\prime}(u,\gamma,A) \geq 2 f \bigg( \frac{1}{2} |[u](\bar{x})|\bigg).
\end{equation}
Since only finite energy approximations are of interest, we consider a sequence $\{(\ue,\gae)\}_{\varepsilon}$ in $\W \times \M$ with $\gae = g_{\varepsilon} \mathcal{L}^1$ for some $g_{\varepsilon} \in L^1(a,b)$ and $E_{\varepsilon}(\ue,\gae) \leq C_0$ for some uniform constant~$C_0$, for all $\varepsilon>0$, such that $\ue \to u$ in $L^1(a,b)$ and $\gae \to \gamma$ in the flat norm. We fix an arbitrary $\delta > 0$ such that $(\bar{x}-2\delta, \bar{x}  +2\delta) \subset A$ (recall that~$A \subset (a,b)$ is open). By definition of $(u(\bar{x}+),u(\bar{x}-))$ and since $\ue \to u$ in measure, we readily find points $\bar{x}_{\varepsilon}^- \in (\bar{x} - \delta, \bar{x})$ and $\bar{x}_{\varepsilon}^+ \in (\bar{x} , \bar{x}+ \delta)$ such that, for sufficiently small $\varepsilon$,
\begin{equation}
\vert \ue(\bar{x}^+_{\varepsilon})- u(\bar{x}+)\vert < \delta  \quad \text{and} \quad  \vert \ue(\bar{x}^-_{\varepsilon})- u(\bar{x}-)\vert < \delta \label{goodappr}
\end{equation}
(also cp.~\cite[Lemma 5.1]{LuVi97}). Using the monotonicity of $A \mapsto E_\varepsilon (\ue,\gae,A)$ and applying the estimate~\eqref{eq_f_integral_mv} with $(a,b)$ replaced by $(\bar{x} - 2\delta, \bar{x} + 2\delta)$ on an associated grid of points~$\bar{\mathscr{G}}_{\varepsilon}$, we obtain from the subadditivity and non-negativity of~$f$ 
\begin{align}
 E_{\varepsilon}({\ue},{\gae},A)&\geq2\sum_{x_\alpha \in \bar{\mathscr{G}}_{\varepsilon}} f \bigg( \varepsilon  \dashint_{I_\varepsilon(x_{\alpha}) }{ |g_\varepsilon| \diff t } \bigg) \notag\\
 & \geq 2  f \bigg(\sum_{x_\alpha \in \bar{\mathscr{G}}_{\varepsilon}}\frac{1}{2}  \int_{I_\varepsilon(x_\alpha)}{ |{g_{\varepsilon}}| \dd x}\bigg) \geq 2 f \bigg( \frac{1}{2} \int_{(\bar{x}_{\varepsilon}^- ,\bar{x}_{\varepsilon}^+)}{ |g_{\varepsilon}|  \dd x} \bigg), \label{inseration1}
\end{align} 
for all $\varepsilon < \delta$ (notice the inclusion $(\bar{x}_{\varepsilon}^- ,\bar{x}_{\varepsilon}^+) \subset \bigcup\{I_\varepsilon(x_\alpha) \colon x_\alpha \in \bar{\mathscr{G}}_{\varepsilon}\}$). For the argument on the right-hand side, we observe from the Cauchy--Schwarz inequality, the inequalities in~\eqref{goodappr} and from the energy bound
\begin{align*}
  \int_{(\bar{x}_{\varepsilon}^- ,\bar{x}_{\varepsilon}^+)}{ |g_{\varepsilon}|  \dd x} 
   & \geq \Bigg[  \int_{(\bar{x}_{\varepsilon}^- ,\bar{x}_{\varepsilon}^+)}{ |{u}^{\prime}_{\varepsilon}|  \dd x} - \int_{(\bar{x}_{\varepsilon}^- ,\bar{x}_{\varepsilon}^+)}{ |{g_{\varepsilon}} -{u}^{\prime}_{\varepsilon}| \dd x} \Bigg]_+ \\
   & \geq \Big[  |{u}_{\varepsilon} (\bar{x}_{\varepsilon}^+) - u_{\varepsilon}(\bar{x}_{\varepsilon}^-) | -  |\bar{x}_{\varepsilon}^+ - \bar{x}_{\varepsilon}^-|^{\frac{1}{2}} \norm{\ggae - \ue^{\prime}}_{L^2(a,b)}  \Big]_+ \\
   & \geq \Big[ |u(\bar{x}+)-u(\bar{x}-)| -2 \delta - (2 \delta)^{\frac{1}{2}} C_0^{\frac{1}{2}} \Big]_+.
\end{align*}
Using once again the fact that~$f$ is increasing, we can continue to estimate~\eqref{inseration1} from below via
\begin{equation*}
 E_{\varepsilon}({\ue},{\gae},A) \geq  2  f  \left( \frac{1}{2} \Big[ |[u](\bar{x})| -2 \delta - C_0^{\frac{1}{2}} \, (2 \delta)^{\frac{1}{2}} \Big]_+ \right).
\end{equation*}
Now, passing to the $\liminf$ as $\varepsilon \to 0$ and then letting $\delta \to 0^+$, we obtain~\eqref{eqn_jump_point_lower_estimate}.

\emph{Step 2.} For an arbitrary $M \in \N$ with $M \leq \#(J_u \cap A)$ we select a set $\lbrace x_1, \dots, x_M \rbrace$ containing~$M$ points of $J_u \cap A$ and pairwise disjoint open intervals $I_1, \dots , I_M$ in $A$ such that $x_i \in I_i$ for all $i = 1, \dots, M$. First, we apply the monotonicity and superadditivity of $E^{\prime}(u,\gamma, \, \cdot \,)$ (see Remark~\ref{remark_properties_lower_limit}) and then the estimate of Step 1 for~$I_i$ instead of $A$. This yields
\begin{equation*}
E^{\prime}(u,\gamma, A ) \geq \sum_{i = 1}^{M} E^{\prime}(u,\gamma,I_i) \geq  \sum_{i = 1}^{M} 2  f  \left( \frac{1}{2} |[u](x_i)| \right).
\end{equation*}
Since $M$ is arbitrary and $J_u$ is at most countable, the claim of the proposition follows.
\end{proof}

\subsection{Estimate from below of the volume and Cantor terms}
\label{subsec: estimate from below volume part 1d}

We basically follow the idea of Lussardi and Vitali from \cite[Lemma 4.3 and Lemma 4.4]{LuVi97}. We start by proving that approximation sequences in $W^{1,1}(a,b) \times \M$ can be modified in such a way that in the limit we additionally have the optimal $L^\infty$-estimate.

\begin{lemma}\label{LVtrunE}
Let $(u,\gamma) \in  BV(a,b) \times \mathcal{M}(a,b)$  with $\norm{u}_{L^{\infty}(a,b)} \leq K$, $\gamma = D^su+ g \mathcal{L}^1$ and $u^{\prime}- g \in L^2(a,b)$. Furthermore, let $\{(u_{\varepsilon},\gamma_\varepsilon)\}_{\varepsilon}$ be a sequence in $W^{1,1}(a,b) \times \M$ with $\norm{\ue}_{L^{\infty}(a,b)} \leq K$, $\gae = g_{\varepsilon} \mathcal{L}^1$  and $\ue^{\prime}- \ggae \in L^2(a,b)$ for all $\varepsilon>0$ such that $\{\ue\}_{\varepsilon}$ converges to~$u$ a.e.~in $(a,b)$ and in $L^1(a,b)$ and such that $\{\gae\}_{ \varepsilon}$ converges to~$\gamma$ in the flat norm. There exists a sequence $\{(\bar{u}_{\varepsilon}, \bar{\gamma}_{\varepsilon})\}_{\varepsilon}$ in $W^{1,1}(a,b) \times \mathcal{M}(a,b)$  with $\norm{\bue}_{L^{\infty}(a,b)} \leq K$, $\bar{\gamma}_{\varepsilon}= \bar{g}_{\varepsilon} \mathcal{L}^1$ for some $\bar{g}_{\varepsilon} \in L^1(a,b)$ and $E_{\varepsilon}(\bar{u}_{\varepsilon}, \bar{\gamma}_{\varepsilon}) \leq E_{\varepsilon}(\ue, \gae)$ for all $\varepsilon$ such that $\{\bar{u}_{\varepsilon}\}_{\varepsilon}$ converges to~$u$ a.e.~in $(a,b)$ and in $L^q(a,b)$ for all $1 \leq q < \infty$ and
\begin{equation*}
\limsup_{\varepsilon \to 0} \norm{\bar{u}_{\varepsilon} - u}_{L^{\infty}(a,b)} \leq \sup\big\{ |[u](x)| \colon x \in J_u \big\}.
\end{equation*}
If, in addition, the energies $\{E_{\varepsilon}(\ue, \gae)\}_{\varepsilon}$ are bounded, then $\{\bar{\gamma}_{\varepsilon}\}_{\varepsilon}$ converges to~$\gamma$ in the flat norm.
\end{lemma}

\begin{proof}
We follow the outline of the proof for \cite[Lemma 4.3]{LuVi97}, which, however, needs some modifications due to the additional variable $\gamma$. In what follows, we may assume that the precise representatives of~$u$ and~$\ue$ for each~$\varepsilon$ are considered. The function~$u$ can be decomposed as $u_a+u_j+u_c$, see~\eqref{decomposition_u_1d}, where~$u_j$ is a jump function with jump discontinuities at any point of~$J_u$ and where $u_a+u_c$ is uniformly continuous in $(a,b)$. We set 
\begin{equation*}
 \sigma  \coloneqq  \sup\big\{ |[u_j](x)| \colon x \in J_u \big\}
\end{equation*}
and first claim that for every $n \in \N$ there exists $\delta_{n} \in (0,\tfrac{1}{n}]$ such that
\begin{align}
x, y \in (a,b) \text{ with }|x-y| < \delta_{n} \quad \Rightarrow \quad & |u_j(x)-u_j(y)| < \sigma + \frac{1}{n},  \label{estimatediameter} \\
 & | u_a(x) + u_c(x) - u_a(y) - u_c(y)| < \frac{1}{n} \label{smallabscont_cantor} .
\end{align}
In fact, there are only finitely many points $\bar{x}_1, \dots, \bar{x}_{m(n)}$ in $J_u$ that have to be excluded to deduce that
\begin{align}
 \sum_{z \in J_u \setminus \lbrace\bar{x}_1, \dots, \bar{x}_{m(n)}\rbrace}|[u_j](z)| < \frac{1}{n}. \label{smalljumps}
\end{align}
By choosing~$\delta_{n}>0$ sufficiently small we can guarantee~\eqref{estimatediameter} by~\eqref{smalljumps} and the definition of~$\sigma$ provided that each interval of length $\delta_{n}$ contains at most one $\bar{x}_j$, and we can further ensure~\eqref{smallabscont_cantor}, by the uniform continuity of $u_a + u_c$. We then consider a partition~$P_n$ of $(a,b),$ i.e.,
\begin{equation*}
a=x_0 < x_1 < \ldots <x_{k} < x_{k+1} = b
\end{equation*}
(where the dependence of the points on~$n$ is not written explicitly) such that the mesh size is less than~$\delta_n$, i.e., $x_{i+1} -x_{i} < \delta_n$ for all $i \in \{ 0, \ldots, k\}$, $x_i \notin J_u$ and $u_{\varepsilon}(x_i) \to u(x_i)$ as $\varepsilon \to 0$ for every $i \in\{ 1, \dots, k\}$, which is possible by the pointwise convergence $\ue \to u$ a.e.~in~$(a,b)$. Since by construction of~$\delta_n$ at most one of the points $\bar{x}_1, \dots, \bar{x}_{m(n)} \in J_u$, where a large jump of~$u_j$ occurs, may belong to the interval $[x_i,x_{i+1}]$, we necessarily have $|u_j(x) - u_j(y)| < \tfrac{1}{n}$ for $y = x_i$ or $y=x_{i+1}$ such that as a consequence of~\eqref{smallabscont_cantor} there holds
\begin{equation}
u(x) \in \left[ \min \{ u(x_{i}) ,  u(x_{i+1})\} - \frac{2}{n}, \max \{ u(x_{i}) ,  u(x_{i+1})\}  +  \frac{2}{n}\right]   \label{estiamteeachintervall}
\end{equation}
for all  $x \in [x_{i}, x_{i+1}]$ and every $i \in \{ 1, \ldots, k-1\}$. 

After having fixed the partitions~$P_n$, we can now start with the construction of the sequence  $\{\bue\}_{\varepsilon}$. Since $\ue \to u$ in measure and $u_{\varepsilon}(x_i) \to u(x_i)$ for every $i \in\{ 1, \dots, k\}$, we can fix a ``level'' $\bar{\varepsilon}_n$ for each $n \in \N$  such that
\begin{equation}
\Big| \Big\{ x \in (a,b) \colon |\ue(x) -u(x)| \geq \frac{1}{n} \Big\} \Big| \leq \frac{1}{n} \quad \text{for all  } \varepsilon < \bar{\varepsilon}_n \label{kappaarbi}
\end{equation}
and
\begin{equation}
|u_{\varepsilon}(x_i)-u(x_i)| < \frac{1}{n} \quad \text{for every } i=1, \dots , k  \text{ and all } \varepsilon < \bar{\varepsilon}_n. \label{estimatemeshpoints}
\end{equation}
Notice that we can choose $\{\bar{\varepsilon}_n\}_n$ strictly decreasing and such that $\bar{\varepsilon}_n \to 0^+$ as $n \to \infty$. For $\varepsilon > \bar{\varepsilon}_1$ we then set $\bue  \coloneqq  \ue$ and $\bar{\gamma}_{\varepsilon}  \coloneqq  \gamma_{\varepsilon}$. Otherwise, if $\varepsilon \in (0, \bar{\varepsilon}_1]$, we first determine the unique $n = n(\varepsilon) \in \N$ with $ \varepsilon \in (\bar{\varepsilon}_{n+1}, \bar{\varepsilon}_{n}]$. On the first and the last interval of~$P_n$ we set $\bue$ equal to  $\ue(x_1)$ and $\ue(x_k)$, respectively. On an arbitrary interior interval $[\alpha, \beta]$ of the form $[x_i, x_{i+1}]$ for some $i \in \{1, \dots, k-1\}$ we define, after assuming without loss of generality $\ue(\alpha) \leq \ue(\beta)$,
\begin{equation*}
\bue(x) : = \min \left\lbrace  \max \left\lbrace\ue(x) , \ue(\alpha) - \frac{4}{n}   \right\rbrace , \ue(\beta) + \frac{4}{n} \right\rbrace
\end{equation*}
for every $x \in [\alpha, \beta]$, which is the projection of~$\ue$ onto $[\ue(\alpha) - 4/n ,\ue(\beta) + 4/n]$. Because of $\bue(x_i)=\ue(x_i)$ for every $i \in \{1, \dots, k\}$ and $\ue \in W^{1,1}(a,b)$, we clearly have $\bue \in W^{1,1}(a,b)$. Moreover, the $L^\infty$ bound on~$\ue$ with constant~$K$ directly transfers to~$\bue$. 

We next study the asymptotic behavior of the sequence $\{\bue\}_{\varepsilon}$. From the definition of~$\bue$ and with~\eqref{estimatemeshpoints}, we observe for all $\varepsilon \leq \bar{\varepsilon}_1$
\begin{equation*}
\bue(x) \in \left[ \min \{ u(x_{i}) ,  u(x_{i+1})\} - \frac{5}{n}, \max \{ u(x_{i}) ,  u(x_{i+1})\}  +  \frac{5}{n}\right]
\end{equation*}
which, via~\eqref{estimatediameter} and~\eqref{smallabscont_cantor}, implies
\begin{equation*}
|\bue(x) - u(x)| \leq  \sigma + \frac{7}{n},
\end{equation*}
for all  $x \in [x_{i}, x_{i+1}]$ and every $i \in \{ 1, \ldots, k-1\}$. Since the latter estimate is also satisfied for the first and the last interval of the partition, we then infer from $n=n(\varepsilon) \to \infty$ as $\varepsilon \to 0$ the estimate
\begin{equation*}
\limsup_{\varepsilon \to 0}\norm{\bue - u}_{L^{\infty}(a,b)} \leq \sigma = \sup\big\{ |[u](x)| \colon x \in J_u \big\}.
\end{equation*}
In order to show the convergence claims of $\{\bue\}_{\varepsilon}$, we again consider an arbitrary interior interval $[\alpha,\beta]$ of the partition~$P_n$. If we denote the pointwise projection of $\ue$ onto $[\ue(\alpha)- 3/n, \ue(\beta)+3/n]$ by $\ue^*$, we have
\begin{equation*}
|\bue - u| \leq |\bue - \ue^*| + |\ue^* - u| \leq \frac{1}{n} + |\ue - u|\quad  \text{ in } [\alpha, \beta],
\end{equation*}
as we know $u(x) \in [\ue(\alpha)- 3/n, \ue(\beta)+3/n]$ due to~\eqref{estiamteeachintervall} and~\eqref{estimatemeshpoints}. Since the length of the first and last interval of~$P_n$ vanish in the limit $n \to \infty$ and hence for $\varepsilon \to 0$, this implies the pointwise convergence of $\{\bue\}_{\varepsilon}$ to~$u$ a.e.~on $(a,b)$. In addition, as $\{\bue\}_{\varepsilon}$ is bounded in $L^\infty(a,b)$, convergence in $L^q(a,b)$ for  all $1 \leq q < \infty$ follows from the dominated convergence theorem. For later purposes, we notice from the definition of~$\bar{u}_{\varepsilon}$ and the previous inclusion for~$u$ that for every $\varepsilon$ we have 
\begin{multline}
\{ x \in (a,b) \colon u_{\varepsilon}(x)\ne\bar{u}_{\varepsilon}(x) \}  \\
 \subset \Big\{ x \in (a,b) \colon |\ue(x) -u(x)| \geq \frac{1}{n} \Big\}  \cup (x_0,x_1) \cup (x_k,x_{k+1}). \label{setinclusion} 
\end{multline}

We next define the sequence $\{\bar{\gamma}_{\varepsilon}\}_{\varepsilon}$ in~$\M$ by setting for every $\varepsilon$ 
\begin{align*}
\bar{g}_{\varepsilon}(x): = \left\{
\begin{array}{l l}
g_{\varepsilon}(x)& \quad \text{if } u_{\varepsilon}(x)=\bar{u}_{\varepsilon}(x),\\
0 & \quad \text{otherwise},
\end{array}
\right.
\end{align*} 
and $\bar{\gamma}_{\varepsilon}  \coloneqq  \bar{g}_{\varepsilon} \mathcal{L}^1$.
Since there holds $\bue^{\prime} = 0$ on $\{x \in (a,b) \colon  \bue(x) \neq \ue(x)\}$ and $\bue^{\prime} = {u}_{\varepsilon}'$ on $\{x \in (a,b) \colon  \bue(x) = \ue(x)\}$, it follows that 
\begin{align*}
E_{\varepsilon}(\bar{u}_{\varepsilon}, \bar{\gamma}_{\varepsilon}) & = \int_{a}^b {|\bar{u}_{\varepsilon}'(x)- \bar{g}_{\varepsilon}(x)|^2 \diff x} + \frac{1}{\varepsilon} \int_a^b f \bigg( \varepsilon \dashint_{I_\varepsilon(x) \cap (a,b)}{ |\bar{g}_{\varepsilon}(t)| \diff t } \bigg) \diff x \\
 & \leq \int_a^b {|{u}_{\varepsilon}'(x)- {g}_{\varepsilon}(x)|^2 \diff x} +  \frac{1}{\varepsilon} \int_a^b f \bigg( \varepsilon \dashint_{I_\varepsilon(x) \cap (a,b)}{ |g_{\varepsilon}(t)| \diff t } \bigg) \diff x =  E_{\varepsilon}(\ue, \gae). 
\end{align*}

If, in addition, $\{E_{\varepsilon}(\ue, \gae)\}_{\varepsilon}$ is a bounded sequence,
then $\{\bar{u}_{\varepsilon}^{\prime} - \bar{g}_{\varepsilon}\}_{\varepsilon}$ is a bounded sequence in $L^2(a,b)$. In view of~\eqref{eq:flat_trivial_estimate} and the Cauchy--Schwarz inequality we then notice for $1 < q < \infty$ 
\begin{align*}
 \normf{\bar{\gamma}_{\varepsilon} - \gae} & \leq \normf{D\bar{u}_{\varepsilon} - Du_{\varepsilon}} + \normf{(\bar{\gamma}_{\varepsilon} - D\bar{u}_{\varepsilon}) - (\gae - Du_{\varepsilon}) } \\
 & \leq \norm{\bar{u}_{\varepsilon} - u_{\varepsilon}}_{L^1(a,b)} +  \normf{(g_{\varepsilon} - u_{\varepsilon}') \1_{\{ u_{\varepsilon} \neq \bar{u}_{\varepsilon} \}} \mathcal{L}^1} \\
 & \leq \norm{\bar{u}_{\varepsilon} - u_{\varepsilon}}_{L^1(a,b)} + \norm{(g_{\varepsilon} - u_{\varepsilon}') \1_{\{ u_{\varepsilon} \neq \bar{u}_{\varepsilon} \}} }_{L^1(a,b)} \\
 & \leq \norm{\bar{u}_{\varepsilon} - u_{\varepsilon}}_{L^1(a,b)} + \norm{g_{\varepsilon} - u_{\varepsilon}'}_{L^2(a,b)} |\{ x \in (a,b) \colon  u_{\varepsilon}(x) \neq \bar{u}_{\varepsilon}(x) \}|^{\frac{1}{2}}. 
\end{align*}
If we consider the limit $\varepsilon \to 0$ on the right-hand side, the first term disappears, since we have the convergence $\bar{u}_{\varepsilon} - u_{\varepsilon} \to 0$ in $L^1(a,b)$, while the second term disappears by~\eqref{setinclusion} combined with~\eqref{kappaarbi} and the fact that the length of the intervals in~$P_n$ vanish for $\varepsilon \to 0$. Therefore, we have $\bar{\gamma}_{\varepsilon} - \gae \to 0$ in the flat norm. Since by assumption there holds $\gae \to \gamma$ in the flat norm, we conclude that we also have $\bar{\gamma}_{\varepsilon} \to \gamma$  in the flat norm, which completes the proof of the lemma.
\end{proof}

For a localization procedure, we further need the following statement on the relation between $E^{\prime}(u, \gamma, I)$ and the $\Gamma$-lower limit $E^{\prime}(u^I, \gamma^I)$, where $I \subset (a,b)$ is an open interval, $u^I$ is the extension of $u|_I$ to $(a,b)$ with inner traces and $\gamma^I$ is the restriction of~$\gamma$ to $I$.

\begin{lemma}\label{cutoff}
Let $I = (\alpha, \beta)$ be an open interval in $(a,b)$. Let $(u,\gamma) \in  BV(a,b) \times \mathcal{M}(a,b)$  with $\norm{u}_{L^{\infty}(a,b)} \leq K$, $\gamma = D^su+ g \mathcal{L}^1$ and $u^{\prime}- g \in L^2(a,b)$. Then for $(u^I,\gamma^I) \in BV(a,b) \times \M$ defined via 
\begin{align*}
u^I(x): = \left\{
\begin{array}{l l}
u(\alpha +) & \quad x \in  (a, \alpha),\\
u(x) & \quad x \in [\alpha,\beta], \\
u(\beta -) & \quad x \in  (\beta, b),
\end{array}
\right.
\end{align*}
and $ \gamma^I: = \gamma \mrs I$ there holds 
\begin{equation*}
E^{\prime}(u^I, \gamma^I) \leq E^{\prime}(u, \gamma, I).
\end{equation*}
\end{lemma} 

\begin{proof}
We proceed analogously to the proof of \cite[Lemma 4.4]{LuVi97}. By definition of~$E'$ as the $\Gamma$-lower limit of $\{E_\varepsilon\}_{\varepsilon}$, there exists a sequence $\{(\ue, \gae)\}_{\varepsilon}$ in $\W \times \M$ with $\norm{\ue}_{L^{\infty}(a,b)} \leq K$ and $\gae= g_{\varepsilon}\mathcal{L}^1$ such that $\ue \to u$ in $L^1(a,b)$, $\gae \to \gamma$
in the flat norm and $\liminf_{\varepsilon \to 0} E_{\varepsilon}(\ue, \gae,I) = E^{\prime}(u,\gamma,I)$. Without loss of generality, we may also assume pointwise convergence $u_\varepsilon \to u$ a.e.~in $(a,b)$. For an arbitrary $\eta \in (0,(\beta-\alpha)/2)$ we then pick points $\alpha_\eta \in (\alpha,\alpha + \eta)$ and $\beta_\eta \in (\beta-\eta,\beta)$ such that on the one hand
\begin{equation}
\label{eqn_localization_1}
 u_\varepsilon(\alpha_\eta) \to u(\alpha_\eta) \quad \text{and} \quad  u_\varepsilon(\beta_\eta) \to u(\beta_\eta) \quad \text{as }  \varepsilon \to 0, 
\end{equation}
and on the other hand
\begin{equation}
\label{eqn_localization_2}
 |u(\alpha_\eta) - u(\alpha+)| + |u(\beta_\eta) - u(\beta-)| < \eta. 
\end{equation}
For $I_\eta  \coloneqq  (\alpha_\eta, \beta_\eta) \subset (\alpha,\beta)$ we now consider the functions $(u^{I_\eta},\gamma^{I_\eta})$ and the sequence $\{(u^{I_\eta}_\varepsilon,\gamma^{I_\eta}_\varepsilon)\}_\varepsilon$ in $\W \times \M$ defined analogously to $(u^I,\gamma^I)$. The condition~\eqref{eqn_localization_1} guarantees $ u^{I_\eta}_\varepsilon \to u^{I_\eta}$ in $L^1(a,b)$, while $\gamma^{I_\eta}_\varepsilon \to \gamma^{I_\eta}$ in the flat norm is trivially satisfied. Moreover, we notice $E_\varepsilon(u^{I_\eta}_\varepsilon,\gamma^{I_\eta}_\varepsilon)= E_\varepsilon(u^{I_\eta}_\varepsilon,\gamma^{I_\eta}_\varepsilon, I) $ for all $\varepsilon < \min\{\alpha_\eta-\alpha,\beta-\beta_\eta\}$. Therefore, we conclude from the definition of~$E_\varepsilon$ that 
\begin{align*}
 E^{\prime}(u^{I_\eta},\gamma^{I_\eta}) \leq \liminf_{\varepsilon \to 0} E_\varepsilon(u^{I_\eta}_\varepsilon,\gamma^{I_\eta}_\varepsilon) & = \liminf_{\varepsilon \to 0} E_\varepsilon(u^{I_\eta}_\varepsilon,\gamma^{I_\eta}_\varepsilon, I)  \\
 & \leq \liminf_{\varepsilon \to 0} E_{\varepsilon}(\ue, \gae,I) = E^{\prime}(u,\gamma,I).
\end{align*}
We next observe $u^{I_\eta} \to u^I$ in $L^1(a,b)$ and $\gamma^{I_\eta} \to \gamma^I$ in the flat norm as $\eta \to 0$, from~\eqref{eqn_localization_2}, respectively, Lemma~\ref{Lemma_weak_negativ_flat} since $\gamma^{I_\eta} \overset{*}{\rightharpoonup} \gamma^I$ in $\M$ by dominated convergence (as we have pointwise convergence $\1_{I_\eta} \to \1_I$ on~$(a,b)$). By the lower semicontinuity of~$E^\prime$ we then arrive at the claim 
\begin{equation*}
 E^{\prime}(u^{I},\gamma^{I}) \leq \liminf_{\eta \to 0} E^{\prime}(u^{I_\eta},\gamma^{I_\eta}) \leq  E^{\prime}(u,\gamma,I). \qedhere
\end{equation*}
\end{proof}

Now, we finally turn to the estimate from below for the volume and the Cantor terms.

\begin{Prop}\label{VPEP}
Let $A$ be an open subset of $(a,b)$. For every $(u,\gamma) \in  BV(a,b) \times \mathcal{M}(a,b)$  with $\norm{u}_{L^{\infty}(a,b)} \leq K$, $\gamma = D^su+ g \mathcal{L}^1$ and $u^{\prime}- g \in L^2(a,b)$ we have
\begin{equation*}
E^{\prime}(u,\gamma,A) \geq \int_A {|u^{\prime} -g|^2 \dd x}  +c_0 \int_A {|g| \dd x} +c_0 |D^c u|(A) .  
\end{equation*}
\end{Prop}

\begin{proof}
\emph{Step 1: With $\sigma \coloneqq  \sup_{x \in J_u} |u(x+)-u(x-)|$, there holds the preliminary estimate}
\begin{equation}
E^{\prime}(u,\gamma) \left( 1 + 3 \sigma \right) \geq \int_a^b |u^{\prime} - g|^2 \dd x   +c_0 \, \int_a^b {|g| } \dd x + c_0|D^cu|(a,b). \label{alstepoo}
\end{equation} 
By definition of~$E'$ as the $\Gamma$-lower limit of $\{E_\varepsilon\}_{\varepsilon}$, there exists a sequence $\{(\ue, \gae)\}_{\varepsilon}$ in $\W \times \M$ with $\norm{\ue}_{L^{\infty}(a,b)} \leq K$ and $\gae= g_{\varepsilon}\mathcal{L}^1$ for some $g_{\varepsilon} \in L^1(a,b)$ for every $\varepsilon >0$ such that $\ue \to u$ in $L^1(a,b)$, $\gae \to \gamma$
in the flat norm and $\liminf_{\varepsilon \to 0} E_{\varepsilon}(\ue, \gae) =E^{\prime}(u,\gamma)$. After assuming without loss of generality $E^{\prime}(u,\gamma) < \infty$ and passing  to a subsequence (not relabeled) and a possible modification of the sequence via Lemma~\ref{LVtrunE} we may further suppose 
\begin{equation}
\lim_{\varepsilon \to 0} E_{\varepsilon}(\ue, \gae) =E^{\prime}(u,\gamma), \quad \text{ in particular } \quad E_{\varepsilon}({u}_{\varepsilon}, {\gamma}_{\varepsilon}) \leq C_0 \quad \text{for all } \varepsilon\label{zusatzschluss}
\end{equation}
for a positive constant~$C_0$ as well as 
\begin{equation*}
\limsup_{\varepsilon \to 0} \norm{u_{\varepsilon} - u}_{L^{\infty}(a,b)} \leq \sigma. 
\end{equation*}
Let $\eta > 0$ be fixed. We may assume that $\norm{\ue - u}_{L^{\infty}(a,b)} \leq \sigma + \eta$ holds for all $\varepsilon$. Analogously as in the proof of Lemma~\ref{LVtrunE} (cf.~\eqref{estimatediameter} and~\eqref{smallabscont_cantor}), there exists $\delta_\eta > 0$ such that $|u(x) - u(y)| < \sigma + \eta$ for all $x,y \in (a,b)$ with $|x-y| < \delta_\eta$. Thus, there holds 
\begin{equation}
x, y \in (a,b) \text{ with }|x-y| < \delta_{\eta} \quad \Rightarrow \quad |\ue(x)-\ue(y)| < 3(\sigma + \eta) \label{estimatediameter_u_eps}
\end{equation}
for all such $\varepsilon$. Next, we apply Lemma~\ref{app} with $u=\ue$ and $\gamma=\gae$. In this way we find a uniform grid~$\Ge$ in the interval $(a,b)$ with grid size~$2 \varepsilon$ such that
\begin{equation}
\int_{\bigcup \{I_\varepsilon(x_\alpha) \colon x_\alpha \in \Gee\}} |g_\varepsilon|\diff t + 2\# (\Ge \setminus \Gee) \\
 \leq \frac{1}{c_0 \varepsilon} \int_a^b f \bigg(\varepsilon \dashint_{I_\varepsilon(x)\cap (a,b)}{ |g_\varepsilon|\diff t } \bigg) \dd x . \label{constructedsequencepropa}
\end{equation}
Let $a_{\ast}= \min \Ge - \varepsilon$ and $b_{\ast}=\mathtoolsset{showonlyrefs=true} \max \Ge +\varepsilon$. We then consider a sequence $\{\tilde{v}_{\varepsilon}\}_\varepsilon$ of functions in $L^\infty(a,b)$, which is defined a.e.~in $(a_\ast,b_\ast)$ by 
\begin{align}
\tilde{v}_{\varepsilon}(x): = \left\{
\begin{array}{l l}
\ue(x) & \quad x \in \bigcup \{I_\varepsilon(x_\alpha) \colon x_\alpha \in \Gee\},\\
\dashint_{I_\varepsilon(x_\alpha)} \ue(z) \dd z & \quad x \in I_\varepsilon(x_\alpha) \text{ for some } x_\alpha \in  \Ge \setminus \Gee,
\end{array} \label{barvepsilon}
\right.
\end{align} 
and then extended to $(a,b)$ by the constant values $\tilde{v}_{\varepsilon}(a_{\ast}^+)$ and $\tilde{v}_{\varepsilon}(b_{\ast}^-)$, respectively, for all~$\varepsilon$. As~$\tilde{v}_{\varepsilon}$ is bounded with $\norm{\tilde{v}_{\varepsilon}}_{L^{\infty}(a,b)} \leq K$ and coincides with~$\ue$ on the set $\bigcup \{I_\varepsilon(x_\alpha) \colon x_\alpha \in \Gee\}$, where 
\begin{align}\label{kleinerRest}
  |(a,b) \setminus \bigcup \{I_\varepsilon(x_\alpha) \colon x_\alpha \in \Gee\}| \leq 2 
\varepsilon  (2 + \#( \Ge \setminus \Gee)) \leq 4 \varepsilon + \varepsilon \frac{C_0}{c_0}
\end{align}
by~\eqref{zusatzschluss} and~\eqref{constructedsequencepropa}, we observe $\tilde{v}_{\varepsilon} \to u$ in $L^q(a,b)$ for all $1 \leq q < \infty$. Moreover, we have ${\tilde{v}_{\varepsilon} \in SBV(a,b)}$ with 
\begin{equation*}
|[\tilde{v}_{\varepsilon}](x)| \leq 3(\sigma + \eta) \quad \text{for all } x \in J_{\tilde{v}_{\varepsilon}},
\end{equation*}
provided that $\varepsilon$ is sufficiently small, i.e., $4 \varepsilon < \delta_\eta$ (such that~\eqref{estimatediameter_u_eps} is satisfied). Since the number of jumps of~$\tilde{v}_{\varepsilon}$ is bounded via~\eqref{constructedsequencepropa} and the definition of~$E_\varepsilon$ by
\begin{equation*}
\# J_{\tilde{v}_{\varepsilon}} \leq 2 \# (\Ge \setminus \Gee) \leq \frac{1}{c_0} E_{\varepsilon}(\ue,\gae), 
\end{equation*}
we end up with the estimate  
\begin{equation}
|D^s\tilde{v}_{\varepsilon}|(a,b) \leq \frac{ 3(\sigma + \eta) }{c_0} E_{\varepsilon}(\ue,\gae) \label{constructedsequenceprop} 
\end{equation}
for the size of $D^s\tilde{v}_{\varepsilon}$ in terms of the energy~$E_{\varepsilon}(\ue,\gae)$. Next, we introduce a sequence $\{\tgae\}_\varepsilon$ of measures in~$\M$, by setting $\tgae \coloneqq  \tilde{g}_{\varepsilon}\mathcal{L}^1$ with $g_\varepsilon \in L^1(a,b)$ defined as 
\begin{align}
\tilde{g}_{\varepsilon}(x)  \coloneqq  \left\{
\begin{array}{l l}
g_{\varepsilon}(x) & \quad x \in \bigcup \{I_\varepsilon(x_\alpha) \colon x_\alpha \in \Gee\},\\
0 & \quad \text{otherwise},
\end{array} \label{bargammadensity}
\right.
\end{align}
for all~$\varepsilon$. In order to show that $\{\tgae +  D^s\tilde{v}_{\varepsilon}\}_{\varepsilon}$ is an approximating sequence of~$\gamma$, we first notice from the definition of~$\tilde{v}_{\varepsilon}$ and $\tgae= \tilde{g}_{\varepsilon}\mathcal{L}^1$ in~\eqref{barvepsilon} and~\eqref{bargammadensity}, with $\tilde{v}_{\varepsilon}' = \ue'$ on $\bigcup \{I_\varepsilon(x_\alpha) \colon x_\alpha \in \Gee\}$ and $\tilde{v}_{\varepsilon}'=0$ on the remaining set of $(a,b)$, that
\begin{align*}
\norm{\gae - (\tgae + D^s\tilde{v}_{\varepsilon})}_{\text{flat}} 
 & = \norm{ g_\varepsilon \1_{(a,b) \setminus \bigcup \{I_\varepsilon(x_\alpha) \colon x_\alpha \in \Gee\}}  \mathcal{L}^1 - D^s\tilde{v}_{\varepsilon}}_{\text{flat}} \\
 &= \norm{(g_\varepsilon - \ue' ) \1_{(a,b) \setminus \bigcup \{I_\varepsilon(x_\alpha) \colon x_\alpha \in \Gee\}}\mathcal{L}^1 +\ue'  \mathcal{L}^1 -\tilde{v}_\varepsilon' \mathcal{L}^1 - D^s\tilde{v}_{\varepsilon}}_{\text{flat}}\\
 &\leq \norm{(g_\varepsilon - \ue' ) \1_{(a,b) \setminus \bigcup \{I_\varepsilon(x_\alpha) \colon x_\alpha \in \Gee\}}\mathcal{L}^1}_{\text{flat}}  + \norm{D \ue  - D\tilde{v}_{\varepsilon}}_{\text{flat}}.
\end{align*}
With~\eqref{eq:flat_trivial_estimate} and the Cauchy--Schwarz inequality we can then continue to estimate
\begin{align*}
& \norm{\gae - (\tgae + D^s\tilde{v}_{\varepsilon})}_{\text{flat}} \\
 & \leq \norm{(g_\varepsilon - \ue' ) \1_{(a,b) \setminus \bigcup \{I_\varepsilon(x_\alpha) \colon x_\alpha \in \Gee\}}}_{L^1(a,b)}  + \norm{\ue  - \tilde{v}_{\varepsilon}}_{L^{1}(a,b)} \\
 & \leq  \norm{g_{\varepsilon} - u_{\varepsilon}'}_{L^2(a,b)} |(a,b) \setminus \bigcup \{I_\varepsilon(x_\alpha) \colon x_\alpha \in \Gee\}|^{\frac{1}{2}} + \norm{\ue  - \tilde{v}_{\varepsilon}}_{L^{1}(a,b)}.
\end{align*}
We now study the terms on the right-hand side. From~\eqref{zusatzschluss} and the definition of~$E_\varepsilon$ we notice that $\{u_{\varepsilon}' -g_{\varepsilon}\}_{\varepsilon}$ is a bounded sequence in $L^2(a,b)$. Together with~\eqref{kleinerRest} and taking into account also the strong convergences $u_\varepsilon \to u$ and $\tilde{v}_{\varepsilon} \to u$ in $L^1(a,b)$, we then arrive at 
\begin{equation*}
 \norm{\gae - (\tgae + D^s\tilde{v}_{\varepsilon})}_{\text{flat}} \to 0 \quad \text{as } \varepsilon \to 0.
\end{equation*}
With $\gae \to \gamma = D^su + g \mathcal{L}^1 $ in the flat norm and Lemma~\ref{Lemma_weak_negativ_flat}, we then conclude 
\begin{equation}
\tgae + D^s\tilde{v}_{\varepsilon} \to \gamma  \quad \text{in the flat norm} \quad \text{and} \quad \tgae + D^s\tilde{v}_{\varepsilon}
\stackrel{\ast}{\rightharpoonup} \gamma  \quad \text{in } \M.
\label{equation_convergence_measure_gamma_modified} 
\end{equation}
For the latter conclusion we have also used the fact that $\{|\tgae + D^s\tilde{v}_{\varepsilon}|(a,b)\}_\varepsilon$ with 
\begin{equation*}
|\tgae + D^s\tilde{v}_{\varepsilon}|(a,b) = \int_a^b |\tilde{g}_{\varepsilon}| \dd x+ |D^s\tilde{v}_{\varepsilon}|(a,b) \quad \text{for every } \varepsilon > 0
\end{equation*}
is a bounded sequence, which is a consequence of the boundedness of $\{\tilde{g}_\varepsilon\}_\varepsilon$ in $L^1(a,b)$ via~\eqref{constructedsequencepropa} and the estimate~\eqref{constructedsequenceprop} (recall also the bound~\eqref{zusatzschluss} on the energies). 

After having discussed the convergence properties of the sequence $\{(\tilde{v}_\varepsilon,\tgae)\}_{\varepsilon}$, we can finally turn to the proof of the estimate~\eqref{alstepoo}. From the definition of~$E_\varepsilon$ we obtain via~\eqref{constructedsequencepropa} and~\eqref{constructedsequenceprop}
\begin{align*}
& E_{\varepsilon}(\ue,\gae)\left( 1 + 3 (\sigma + \eta)\right) \\
  &\geq  \int_a^b |\ue'-g_{\varepsilon}|^2 \dd x + c_0\int_{\bigcup \{I_\varepsilon(x_\alpha) \colon x_\alpha \in \Gee\}} |g_\varepsilon|\dd x  + c_0 |D^s\tilde{v}_{\varepsilon}|(a,b)\\
 &=  \int_a^b |\ue'-g_{\varepsilon}|^2 \dd x + c_0 \int_a^b |\tilde{g}_{\varepsilon}| \dd x+ c_0 |D^s\tilde{v}_{\varepsilon}|(a,b).
\end{align*}
By the choice of the sequence $\{(u_\varepsilon,\gamma_\varepsilon)\}_\varepsilon$ with~\eqref{zusatzschluss} it follows that 
\begin{align*}
& E'(u,\gamma) \left( 1 + 3 (\sigma + \eta)\right) \\
& \geq  \liminf_{\varepsilon \to 0} \int_a^b |\ue'- g_{\varepsilon}|^2 \dd x  + \liminf_{\varepsilon \to 0}\left[  c_0 \int_a^b |\tilde{g}_{\varepsilon}| \dd x+ c_0 |D^s\tilde{v}_{\varepsilon}|(a,b )\right] \\
 &   \geq \int_a^b |u^{\prime} - g|^2 \dd x +c_0 \, \int_a^b {|g| \dd x } + c_0|D^s u|(a,b)  \\
 & \geq \int_a^b |u^{\prime} - g|^2 \dd x  +c_0 \, \int_a^b {|g| \dd x }+ c_0|D^cu|(a,b).
\end{align*}
Let us comment on the second-last inequality. For the first term we first deduce from the boundedness of the sequence $\{u_{\varepsilon}' - g_{\varepsilon}\}_{\varepsilon}$ in $L^2(a,b)$ combined with the convergences $u_\varepsilon' \mathcal{L}^1 \to Du$ and $\gamma_\varepsilon \to \gamma =  D^su + g \mathcal{L}^1$ in the flat norm that $\ue^{\prime}- g_{\varepsilon} \rightharpoonup u^{\prime} - g$ in $L^2(a,b)$ and then employ the lower semicontinuity of the $L^2$-norm with respect to weak convergence in $L^2(a,b)$. For the second and third term we use the weak-$\ast$ convergence $\tgae + D^s\tilde{v}_{\varepsilon} \stackrel{\ast}{\rightharpoonup} \gamma=  D^su + g \mathcal{L}^1$ in~$\M$ from~\eqref{equation_convergence_measure_gamma_modified} and the lower semicontinuity of the total variation with respect to weak-$\ast$ convergence. By the arbitrariness of~$\eta>0$ we conclude from the previous inequality the desired estimate~\eqref{alstepoo}.

\emph{Step 2: Localization.} We fix an arbitrary $\bar{\sigma} >0$ and consider the finite set of points $\{ x_1, \ldots,x_{N-1}\} \subset J_u$ such that $|[u](x_i)| > \bar{\sigma}$ for $i=1,\dots, N-1$. Let $x_0 = a$ and $x_N=b$. Then we have 
\begin{equation*}
\sup_{x \in J_u \cap (x_i,x_{i+1})} |[u](x)| \leq \bar{\sigma} \quad \text{for every } i \in \{0, \dots, N-1\}.
\end{equation*}
For every open subinterval $(\alpha, \beta)$ of $(a,b)$ we consider pairs $(u^{(\alpha,\beta)},\gamma^{(\alpha,\beta)}) \in BV(a,b) \times \M$ defined as in Lemma~\ref{cutoff} as
\begin{align*}
u^{(\alpha, \beta)}(x) = \left\{
\begin{array}{l l}
u(\alpha +) & \quad x \in  (a, \alpha),\\
u(x) & \quad x \in [\alpha,\beta], \\
u(\beta -) & \quad x \in  (\beta, b),
\end{array}
\right.
\end{align*} and $\gamma^{(\alpha, \beta)} = \gamma \mrs (\alpha, \beta).$ 
By Lemma~\ref{cutoff} (with $I=(x_i, x_{i+1})$) and by Step 1 (applied with $u = u^{(x_i,x_{i+1})}$ and $\gamma=\gamma^{(x_i,x_{i+1})})$, we obtain
\begin{align*}
E^{\prime}(u,\gamma, (x_i,x_{i+1})) ( 1 + 3 \bar{\sigma}) 
 & \geq E^{\prime}(u^{(x_i,x_{i+1})},\gamma^{(x_i,x_{i+1})}) ( 1 + 3 \bar{\sigma})\\
 & \geq \int_{x_i}^{x_{i+1}}|u^{\prime} - g|^2 \dd x + c_0 \int_{x_i}^{x_{i+1}} |g| \dd x + c_0|D^cu|(x_i,x_{i+1}) 
\end{align*}
for every $i \in \{0, \dots, N-1\}$. With the superadditivity of $A \mapsto E'(u,\gamma,A)$ from  Remark~\ref{remark_properties_lower_limit}~(i) we then deduce 
\begin{equation*}
 E^{\prime}(u,\gamma)   \left( 1 + 3 \bar{\sigma}\right)   \geq \int_{a}^b|u^{\prime} - g|^2 \dd x + c_0 \int_{a}^{b}  |g| \dd x + c_0|D^cu|(a,b),
\end{equation*}
which, by the arbitrariness of $\bar{\sigma} >0$, leads to
\begin{align*}
E^{\prime}(u,\gamma) &  \geq \int_a^b |u^{\prime} - g|^2 \dd x  +c_0  \int_a^b {|g|  \dd x} + c_0|D^cu|(a,b).
\end{align*}
Applying this estimate to $(u^I,\gamma^I)$, from Lemma~\ref{cutoff} we infer that
\begin{align*}
E^{\prime}(u,\gamma,I) \geq E^{\prime}(u^I,\gamma^I) \geq \lambda(I) 
\end{align*}
for every open interval $I \subset (a,b)$, where $\lambda$ denotes the positive Borel measure on $(a,b)$ that is given by 
\begin{equation*}
\lambda(B) \coloneqq  \int_B |u^{\prime} - g|^2  \dd x +c_0 \, \int_B {|g|} \dd x + c_0|D^cu|(B)
\end{equation*}
for every Borel subset $B$ of $(a,b)$. Therefore, the claim of the proposition follows from Remark~\ref{remark_properties_lower_limit}~(ii).
\end{proof}


\subsection{Conclusion and proof of the {\boldmath\texorpdfstring{$\Gamma$}{Gamma}-\boldmath\texorpdfstring{$\liminf$}{liminf} inequalities}}

For $(u,\gamma) \in  BV(a,b) \times \mathcal{M}(a,b)$  with $\norm{u}_{L^{\infty}(a,b)} \leq K$, $\gamma = D^su+ g \mathcal{L}^1$ and $u^{\prime}- g \in L^2(a,b)$ we have proved so far in Propositions~\ref{JPEP} and~\ref{VPEP} the following lower bounds for the volume, the Cantor and the jump part:
\begin{enumerate}
\item $E^{\prime}(u,\gamma,A) \geq \int_A |u^{\prime} - g|^2  \dd x+ c_0 \int_A |g| \dd x,$
\item $E^{\prime}(u,\gamma,A) \geq c_0  |D^cu|(A), $
\item $E^{\prime}(u,\gamma,A) \geq 2 \sum_{x \in J_u \cap A}  f  \left( \frac{1}{2} |[u](x)|  \right),$
\end{enumerate}
for every open subset $A$ of $(a,b)$. These are now combined to prove the estimate from below of the $\Gamma$-lower limit which shows the first part of Theorem~\ref{mainresult}.
 
\begin{theorem}\label{generalliminf}
For every $(u,\gamma) \in  BV(a,b) \times \mathcal{M}(a,b)$  with $\norm{u}_{L^{\infty}(a,b)} \leq K$, $\gamma = D^su+ g \mathcal{L}^1$ and $u^{\prime}- g \in L^2(a,b)$ we have
\begin{equation*}
E^{\prime}(u,\gamma) \geq E(u, \gamma).  
\end{equation*}
\end{theorem}

\begin{proof}
We consider the Radon measure~$\kappa$ defined by
\begin{equation*}
\kappa(B) \coloneqq  \mathcal{L}^1(B) + \#(J_u \cap B) + \vert D^cu\vert (B)
\end{equation*}
for every Borel subset $B$ of $(a,b)$. Let $C$ be a Borel subset of $(a,b)\setminus J_u$ with $\vert C \vert =0$ such that $\vert D^cu\vert ((a,b)\setminus C)=0$. Then, we obtain
\begin{equation*}
 E^{\prime}(u,\gamma,A)\geq \int_A \psi_i \diff \kappa
\end{equation*}
for $i \in \{1,2,3\}$ and for every open set $A \subset (a,b)$, where
\begin{align*}
 \psi_1 &  \coloneqq  \big( |u^{\prime} - g|^2 + c_0 |g| \big) \1_{(a,b) \setminus (J_u  \cup C)}, \\
 \psi_2 &  \coloneqq  2 \ff{\tfrac{1}{2} |[u]|} \1_{J_u}, \\
 \psi_3 &  \coloneqq  c_0 \1_C. 
\end{align*}
Next, we define
\begin{align*}
\psi(x) \coloneqq \sup_i \psi_i(x) =  \left\{
\begin{array}{l l}
|u^{\prime}(x) -g(x)|^2 +c_0 |g(x)| & \quad \text{if } x \in (a,b) \setminus (J_u  \cup C),  \\ 
2 \, \ff{\frac{1}{2} |[u](x)|}  & \quad \text{if } x \in J_u, \\
c_0 & \quad \text{if } x \in C.
\end{array}
\right.
\end{align*}
By a measure theoretic result (see e.g.~\cite[Lemma 15.2]{br} applied with the set function $\mu(\cdot) \coloneqq E^{\prime}(u,\gamma,\cdot)$) we conclude that 
\begin{equation*}
E^{\prime}(u,\gamma,A) \geq \int_A \sup_i \psi_i \dd \kappa =\int_A \psi \diff \kappa = E(u,\gamma,A)
\end{equation*}
for every open subset $A$ of $(a,b)$. With $A=(a,b)$, this proves the theorem. 
\end{proof}

The $\Gamma$-$\liminf$ in the case $K = \infty$ is a direct consequence of Theorem~\ref{generalliminf}.  

\begin{corollary}
Let $K = \infty$. If $(u_\varepsilon, \gamma_\varepsilon) \to (u, \gamma)$ in $L^0((a, b);\overline{\R}) \times \M$, then 
\begin{equation*}
  \liminf_{\varepsilon \to 0} E_\varepsilon(u_\varepsilon, \gamma_\varepsilon) \ge E(u, \gamma). 
\end{equation*}
\end{corollary}

\begin{proof}
Assuming without loss of generality $E_\varepsilon(u_\varepsilon, \gamma_\varepsilon) < C_0$ for a positive constant~$C_0$, by Theorem~\ref{com-K-infty} we have that $u \in  BV_{\infty,\mathcal{P}}(a,b)$ and there is a partition $a = x_0 < x_1 \ldots < x_{m} = b$ such that $J_u \subset \{x_1, \ldots, x_{m-1}\} \cup \mathcal{F}(u)$ and for $A_\delta \coloneqq \mathcal{F}(u) \cap \bigcup_i (x_{i-1}+\delta, x_i-\delta)$, with $\delta > 0$, we have $\chi_{A_\delta} u_\varepsilon \to \chi_{A_\delta} u$ in $L^1(a,b)$ and $\gamma_\varepsilon \mrs A_\delta \to \gamma \mrs A_\delta$ in the flat norm. From Remark~\ref{rmk:en-lb} and Theorem~\ref{generalliminf} (applied for a suitable $K$) we then get 
\begin{equation*}
  \liminf_{\varepsilon \to 0} E_\varepsilon(u_\varepsilon, \gamma_\varepsilon) 
  \geq 2 c_0(m-1) + E(u,\gamma,A_\delta). 
\end{equation*}
The assertion follows in the limit $\delta \searrow 0$ from the monotone convergence theorem. 
\end{proof}


\section[{Estimate from above of the \texorpdfstring{$\Gamma$}{Gamma}-upper limit}]{Estimate from above of the \boldmath \texorpdfstring{$\Gamma$}{Gamma}-upper limit}\label{sec:upper-limit}
We now turn to the estimate from above of the $\Gamma$-upper limit $E^{\prime \prime}$. Except for the very last paragraph we assume $K < \infty$ in the whole section. We again restrict ourselves to pairs $(u, \gamma)\in BV(a,b) \times \M$ with $\norm{u}_{\LU} \leq K$, $\gamma= D^su+ g \mathcal{L}^1$ and $u^{\prime} - g \in L^2(a,b)$ since the estimates are trivial otherwise. We first show the result for the particular case $u \in SBV^2(a,b)$ and then deduce the general result by approximation.

\begin{Prop}\label{SBVupperbpundestimate}
For every $(u,\gamma) \in  SBV^2(a,b) \times \mathcal{M}(a,b)$ with $\norm{u}_{L^{\infty}(a,b)} \leq K$, $\gamma = D^su+ g \mathcal{L}^1$ and $u^{\prime}- g \in L^2(a,b)$ we have
\begin{equation*}
 E^{\prime \prime}(u,\gamma) \leq E(u,\gamma). 
\end{equation*}
\end{Prop}

\begin{proof}
Since $u \in SBV^2(a,b)\cap L^{\infty}(a,b)$, the jump set is finite, i.e.,~$J_u = \{ x_1, \ldots, x_{N-1} \}$ for some $N \in \N$, and we may further assume by the Sobolev embedding theorem that~$u$ is a piecewise continuous function with one-sided limits $u(x\pm)$ for all $x \in (a,b)$. Thus,~$\gamma$ is of the form 
\begin{equation*}
\gamma = \sum_{i =1}^{N-1} [u](x_i) \delta_{x_i} + g \mathcal{L}^1, 
\end{equation*}
with $g \in L^2(a,b)$. Let $x_0 = a$ and $x_N=b$. We then choose $\varepsilon$ small enough such that
\begin{equation*}
| x_{i+1} - x_i| >  2 \varepsilon^2 + 4 \varepsilon \quad \text{for every } i \in \{0, \dots, N-1\}.
\end{equation*}
We first define $u_{\varepsilon}\in W^{1,1}(a,b)$ nearby the jumps of $u$ by linear interpolation via
\[
\begin{array}{lll}
& u_{\varepsilon}  \coloneqq u & \text{on } (a,b) \setminus \bigcup_{i=1}^{N-1} (x_i - \varepsilon^2 - 2 \varepsilon, x_i +\varepsilon^2 + 2 \varepsilon),\\
& u^{\prime}_{\varepsilon}  \coloneqq  \frac{u(x_i + \varepsilon^2 + 2\varepsilon)-u(x_i- \varepsilon^2 -2 \varepsilon)}{2\varepsilon^2} &\text{on } (x_i -\varepsilon^2 , x_i +\varepsilon^2) \text{ for } i \in \{1,\ldots,N-1\},\\[0.1cm]
& u^{\prime}_{\varepsilon}  \coloneqq  0 & \text{otherwise}
\end{array}
\]
(see the figure below). By construction, we have $\norm{\ue}_{L^{\infty}(a,b)} \leq K$ for all $\varepsilon$, and taking advantage of the fact that~$u$ is only modified on the intervals $(x_i - \varepsilon^2 - 2\varepsilon, x_i +\varepsilon^2 + 2\varepsilon)$ for $i \in \{1,\ldots,N-1\}$, we also have 
\begin{equation*}
\norm{u_{\varepsilon}-u}_{L^1(a,b)} \leq \int_{\bigcup_{i=1}^{N-1} (x_i - \varepsilon^2 - 2\varepsilon, x_i +\varepsilon^2 + 2\varepsilon)} |u_{\varepsilon} - u| \dd x \leq 4 K (N-1) ( \varepsilon^2 + 2\varepsilon),
\end{equation*}
which shows strong convergence of $\{u_{\varepsilon}\}_{\varepsilon}$ to~$u$ in $L^1(a,b)$. 

\begin{figure}[H]
\centering
\begin{tikzpicture}[scale= 0.8]
\draw[->] (-0.2,0) -- (7.3,0) node[right]  {$x$};
\draw[->] (0,-0.2) -- (0,5) node[above]  {$u$} ;
\draw[-] (0,0) --(2,1);
\draw[-] (2,1.5) -- (5, 3);
\draw[-]  (5, 4) -- (7, 5);
\draw[dashed] (2,-0.5) -- (2,5);
\draw[dashed] (5,-0.5) -- (5,5);

\begin{scope}[xshift=9cm,yshift=7cm]

\draw[->] (-0.2,-7) -- (7.3,-7) node[right]  {$x$};
\draw[->] (0,-7.2) -- (0,-2) node[above]  {$u_{\varepsilon}$} ;
\draw[-] (0,-7) --(1,-6.5);
\draw[-](1,-6.5) -- (1.7, -6.5);
\draw[-] (1.7,-6.5) --(2.3, -4.85);
\draw[-] (2.3,-4.85) -- (3.3, -4.85);
\draw[-]  (3.3, -4.85) -- (3.7, -4.65 );
\draw[-]  (3.7, -4.65) -- (4.7, -4.65);
\draw[-]  (4.7, -4.65) -- (5.3, -2.35);
\draw[-]  (5.3, -2.35) -- (6.3, -2.35);
\draw[-]  (6.3, -2.35) -- (7, -2);
\coordinate (A)  at (2,-7.3);
\node[below, scale=0.8, blue] at (A) {$2 \varepsilon^2$};
\coordinate (C)  at (5,-7.3);
\node[below, scale=0.8, blue] at (C) {$2 \varepsilon^2$};
\draw[dotted] (1,-14.3) -- (1,-2);
\draw[dotted] (1.7,-14.3) -- (1.7,-2);
\draw[dotted] (2.3,-14.3) -- (2.3,-2);
\draw[dotted] (3.3,-14.3) -- (3.3,-2);
\draw[dotted] (3.7,-14.3) -- (3.7,-2);
\draw[dotted] (4.7,-14.3) -- (4.7,-2);
\draw[dotted] (5.3,-14.3) -- (5.3,-2);
\draw[dotted] (6.3,-14.3) -- (6.3,-2);

\draw[->] (-0.2,-14) -- (7.3,-14) node[right]  {$x$};
\draw[->] (0,-14.2) -- (0,-9) node[above]  {$g_{\varepsilon}$} ;

\fill[yellow!50]  (1.7,-14) rectangle (2.3, -12);
\draw[-] (1.7,-14) rectangle (2.3, -12);
\fill[yellow!50]   (4.7,-14) rectangle (5.3, -12);
\draw[-]  (4.7, -14) rectangle (5.3, -12);

\fill[green! 70 ! black] (0,-14) rectangle  (1,-13.5);
\draw (0,-14) rectangle  (1,-13.5);
 \fill[green! 70 ! black] (3.3,-14) rectangle  (3.7,-13.5);
\draw (3.3,-14) rectangle  (3.7,-13.5);
\fill[green! 70 ! black]  (6.3,-14) rectangle  (7,-414/30);
\draw (6.3,-14) rectangle  (7,-414/30);

\node[scale=0.8, below,  text=blue] at (2,-14.3) {$2 \varepsilon^2 + 4 \varepsilon$};
\node[scale=0.8, below,  text=blue] at (5,-14.3) {$ 2 \varepsilon^2 + 4 \varepsilon $};

\draw[blue, thick, <->] (1,-14.3) -- (3.3,-14.3);
\draw[blue, thick,<->] (3.7,-14.3) -- (6.3,-14.3);
\draw[blue, thick, <->] (1.7,-7.3) -- (2.3,-7.3);
\draw[blue, thick,<->] (4.7,-7.3) -- (5.3,-7.3);

\node[fill = green! 70! black] (A) at (0.6,-11) {$\gamma^{a}$};
\coordinate (Y3) at (3.5,-13.8);
\coordinate (G1) at (0.5,-13.8);
\coordinate (G2) at (2.75,-13.8);
\coordinate (G3) at (4.25,-13.6);
\coordinate (G4) at (6.6,-13.9);
\draw[->] (A) --(G1);
\draw[->] (A) --(Y3);
\draw[->] (A) --(G4);

\draw[dashed] (2,-2) -- (2,-7);
\draw[dashed] (5,-2) -- (5,-7);

\draw[dashed] (2,-14.3) -- (2,-7.9);
\draw[dashed] (5,-14.3) -- (5,-7.9);

\node[fill = yellow!50] (D) at(5,-11) {$ \frac{u(x_i+ \varepsilon^2 + 2\varepsilon) - u(x_i- \varepsilon^2 - 2\varepsilon)}{2 \varepsilon^2}$};
\coordinate (Y1) at (2,-12.5);
\coordinate (Y2) at (5,-12.5);

\draw[->] (D) --(Y1);
\draw[->] (D) --(Y2);

\end{scope}

\node[text width=7cm, align=justify,scale=0.8] at (3.5,-1.8) {{\bf Fig.:} Construction of the recovery sequence for a piecewise affine function with jump discontinuities.};

\end{tikzpicture}
\end{figure}

We next define $\gamma_{\varepsilon} \in \mathcal{M}(a,b)$ as $\gae = g_{\varepsilon} \mathcal{L}^1$, where $\ggae \in L^2(a,b)$ is given by
\[g_{\varepsilon}: =\begin{cases} g & \text{on } (a,b) \setminus \bigcup_{i=1}^{N-1} (x_i - \varepsilon^2 - 2\varepsilon, x_i +\varepsilon^2 + 2\varepsilon),\\ 
\ue^{\prime}& \text{on } (x_i -\varepsilon^2 , x_i +\varepsilon^2) \text{ for } i \in \{1, \ldots,N-1\},
 \\0& \text{otherwise}.
\end{cases}\]
We notice that 
\begin{equation*}
 \gamma_{\varepsilon}-\gamma = \sum_{i=1}^{N-1} \Big( \ue^{\prime}\1_{(x_i -\varepsilon^2 , x_i +\varepsilon^2)} \mathcal{L}^1 -  [u](x_i) \delta_{x_i} - g \1_{ (x_i - \varepsilon^2 -2 \varepsilon, x_i +\varepsilon^2 + 2\varepsilon)} \mathcal{L}^1 \Big).
\end{equation*}
Therefore, we observe from the definition of $u_\varepsilon^{\prime}$ that for every function $\varphi \in W^{1, \infty}_0(a,b)$ with $\norm{\varphi}_{W^{1, \infty}_0(a,b)}\leq 1$ there holds
\begin{align*}
 \int_{x_i-\varepsilon^2 - 2\varepsilon}^{x_i +\varepsilon^2 + 2\varepsilon} \varphi \diff (\gamma_{\varepsilon}-\gamma) & = \big[ u(x_i+\varepsilon^2 + 2\varepsilon)-u(x_i-\varepsilon^2 -2 \varepsilon) \big] \bigg( \, \dashint_{x_i -\varepsilon^2}^{x_i +\varepsilon^2} \varphi \diff x- \varphi(x_i)\bigg) \\
 &\quad +  \big[ u(x_i+\varepsilon^2 + 2\varepsilon)-u(x_i+) - u(x_i-\varepsilon^2 - 2\varepsilon)+u(x_i-)\big]  \varphi(x_i) \\
 & \quad - \int_{x_i-\varepsilon^2 - 2\varepsilon}^{x_i +\varepsilon^2 + 2\varepsilon} \varphi g \diff x
\end{align*}
for every $i \in \{1,\ldots,N-1\}$. Since we also have $|\varphi(x)-\varphi(x_i)| \leq |x-x_i| \leq \varepsilon^2$ on $(x_i -\varepsilon^2 , x_i +\varepsilon^2)$, we deduce from the bound $\norm{u}_{L^{\infty}(a,b)} \leq K$ and the Cauchy--Schwarz inequality
\begin{align*}
 \norm{\gamma_{\varepsilon}-\gamma}_{\text{flat}} & \leq 2K (N-1)\varepsilon^2  \\
 & \quad +  \sum _{i =1}^{N-1} \big[ |u(x_i+\varepsilon^2 + 2\varepsilon)-u(x_i+)| +|u(x_i-\varepsilon^2 - 2\varepsilon)-u(x_i-)|\big] \\
 & \quad + \norm{g}_{L^2(a,b)} \big[ 2(N-1) (\varepsilon^2 + 2\varepsilon) \big]^{\frac{1}{2}}.
\end{align*}
This shows the convergence of $\{\gamma_{\varepsilon}\}_{\varepsilon}$ to~$\gamma$ in the flat norm. It only remains to establish the energy estimate. From the construction of $(u_\varepsilon,\gamma_\varepsilon)$, we clearly have $|u'_{\varepsilon}- g_{\varepsilon}| \leq |u' - g|$ on $(a,b)$. Therefore, the elastic energy contribution in $E_{\varepsilon}(u_{\varepsilon},\gamma_{\varepsilon})$ is estimated by 
\begin{equation*}
 \int_a^b |u'_{\varepsilon}- g_{\varepsilon}|^2  \dd x \leq \int_a^b  |u^{\prime} - g|^2 \dd x.
\end{equation*}
Due to the monotonicity of~$f$ and $f(t) \leq c_0 t$ for all $t \geq 0$, we estimate the non-local energy term by 
\begin{align*}
 & \frac{1}{\varepsilon} \int_a^b f \bigg(  \varepsilon \dashint_{I_{\varepsilon}(x)\cap (a,b)}{ |g_{\varepsilon}| \dd t }\bigg) \dd x \\
 & \leq \frac{1}{\varepsilon} \int_{(a,b)\setminus \bigcup_{i=1}^{N-1} (x_i - \varepsilon^2 -  \varepsilon, x_i +\varepsilon^2 + \varepsilon)} f\bigg( \varepsilon \dashint_{I_\varepsilon(x)\cap (a,b)}|g | \dd t\bigg) \dd x \\ 
 & \quad + \frac{1}{\varepsilon} \int_{\bigcup_{i=1}^{N-1} (x_i - \varepsilon^2 -  \varepsilon, x_i +\varepsilon^2 + \varepsilon)} f\bigg(  \frac{1}{2} \int_{(x_i-\varepsilon^2,x_i+\varepsilon^2)} |u_{\varepsilon}^{\prime} | \dd t \bigg) \dd x \\ 
 & \leq c_0 \int_a^b \dashint_{I_\varepsilon(x)\cap (a,b)}|g | \dd t \dd x \\
 & \quad + \frac{2(\varepsilon^2 + \varepsilon)}{\varepsilon} \sum_{i=1}^{N-1} \dashint_{x_i-\varepsilon^2-\varepsilon}^{x_i+\varepsilon^2+\varepsilon}{f\bigg(  \frac{1}{2} |u(x_i+\varepsilon^2 + 2 \varepsilon)-u(x_i-\varepsilon^2- 2\varepsilon)|\bigg)} \dd x .
\end{align*}
With the continuity of~$u$ outside of the jump set~$J_u$ we can pass to the limit $\varepsilon \to 0$ on the right-hand side. In this way, we finally arrive at
\begin{align*}
E^{\prime \prime}(u,\gamma) & \leq \limsup_{\varepsilon \to 0} E_{\varepsilon}(u_{\varepsilon},\gamma_{\varepsilon})\\ &\leq   \int_a^b  | u^{\prime} - g|^2 \dd x + c_0 \int_a^b|g| \dd x + 2 \sum_{i=1}^{N-1} f \bigg( \frac{1}{2}|[u](x_i)| \bigg) \\
& = E(u, \gamma). \qedhere
\end{align*} 
\end{proof}

\begin{remark}\label{rmk:W-Wrelax}
For a general stored energy function $W$ as described in Remark~\ref{rmk:gen-stored-en} the above argument can be augmented with a standard relaxation step by adding to $u_\varepsilon$ a function $v_\varepsilon \in W^{1,p}_0(a,b)$ such that $v_\varepsilon \to 0$ in $L^p(a,b)$ and 
\begin{equation*}
 \limsup_{\varepsilon \to 0} \int_a^b W(u'_{\varepsilon} + v'_{\varepsilon} - g_{\varepsilon}) \dd x 
 \leq \int_a^b  W^{**}(u^{\prime} - g) \dd x.
\end{equation*}
So also in this case we have $E^{\prime \prime}(u,\gamma) \leq \limsup_{\varepsilon \to 0} E_{\varepsilon}(u_{\varepsilon}+v_{\varepsilon},\gamma_{\varepsilon}) \leq E(u,\gamma)$. 
\end{remark}

By approximation with $SBV^2$-functions, we can now give the proof of the second part of Theorem~\ref{mainresult}.

\begin{theorem}\label{thm:Gmma-limsup}
For every $(u,\gamma) \in  BV(a,b) \times \mathcal{M}(a,b)$  with $\norm{u}_{L^{\infty}(a,b)} \leq K$, $\gamma = D^su+ g \mathcal{L}^1$ and $u^{\prime}- g \in L^2(a,b)$ we have
\begin{equation*}
E^{\prime \prime}(u,\gamma) \leq E(u,\gamma).
\end{equation*}
\end{theorem}

\begin{proof}
We here want to construct a sequence $\{(\hat{u}_h,\hat{\gamma}_h)\}_{h}$ in $SBV^2(a,b) \times \M $ (with $\norm{\hat{u}_h}_{L^{\infty}(a,b)} \leq K$, $\hat{\gamma}_h = D^s \hat{u}_h+ \hat{g}_{h} \mathcal{L}^1$ and $\hat{u}_h^{\prime} - \hat{g}_{h} \in L^2(a,b)$ for every $h>0$) such that $\hat{u}_h \to u$ in $L^1(a,b)$, $\hat{\gamma}_h \to \gamma$ in the flat norm and
\begin{equation}
\label{recovery_approximation}
\liminf_{h \to 0}E(\hat{u}_h, \hat{\gamma}_h) \leq E(u,\gamma).
\end{equation}
This is indeed sufficient since by lower semi-continuity of the $\Gamma$-upper limit~$E''$ and by Proposition~\ref{SBVupperbpundestimate} we then conclude with
\begin{equation*}
 E^{\prime \prime}(u,\gamma) \leq \liminf_{h \to 0} E''(\hat{u}_h, \hat{\gamma}_h) \leq  \liminf_{h \to 0} E(\hat{u}_h, \hat{\gamma}_h) \leq E(u,\gamma).
\end{equation*}

We first recall that in dimension one every function $u \in BV(a,b)$ can be represented as $u_a+u_j+u_c$, see~\eqref{decomposition_u_1d} where $u_a \in W^{1,1}(a,b)$, $u_j$ is a pure jump function and $u_c$ is a Cantor function. This allows us to modify the three parts of $u$ separately. We start with the jump function $u_j$. We define $u_{j,h}$ by
\begin{equation*}
u_{j, h}(x): = u_j(a+) + \sum_{ y \in J_{u_{j}}(h) \cap (a,x] } [u_j](y) \quad \text{for } x \in (a,b), 
\end{equation*}
where $J_{u_{j}}(h)  \coloneqq  \{y\in J_{u_j} \colon |[u_j](y)| > h \} = J_{u_{j,h }}$. We observe $\# J_{u_{j,h }} < \infty$, $u_{j,h } \to u_j$ in $L^1(a,b)$ as $h \to 0$, and for all~$h$ the estimate
\begin{equation}
 \sum_{x \in J_{u_{j,h }}}  \ff{ \frac{1}{2} |[u_{j,h }](x)|} \leq \sum_{x \in J_{u_j}} \ff{ \frac{1}{2}|[u_j](x)|}. \label{jumpinequalty}
\end{equation}
For the Cantor function~$u_c$ we use the density of smooth functions in $BV(a,b)$ with respect to the strict topology. In this way, we find a sequence $\{u_{c,h }\}_{h }$ in $W^{1,1}(a,b)$ with $u_{c,h} \to u_c$ in $L^1(a,b)$ and
\begin{align}
\lim_{h \to 0} \int_{a}^{b} |u^{\prime}_{c,h }| \dd x = |Du_c|(a,b). \label{cantorinequalty}
\end{align}
The absolutely continuous part~$u_a$ is first extended to a $W^{1,1}(\R)$ function with compact support and we then set $u_{a, h } \coloneqq u_a * \psi_{h }$ for all $h >0$, where $\psi_h$ is a standard $h$-mollifier given by $\psi_{h}(x) \coloneqq  h^{-1} \psi (h^{-1}x)$ for $x \in \R$, for a fixed non-negative, symmetric function $\psi \in C^{\infty}(\R)$ with compact support and normalized to $\int_\R \psi \dd x = 1$. We then have $ u_{a,h } \in C^{\infty}(\R)$ for all $h>0$, $u_{ a, h } \to  u_a$ in $L^1(a,b)$ and $u_{a,h }^{\prime} = u_a^{\prime} * \psi_h$, see e.g.~\cite[Theorem~4.2.1]{eg}. Then we set
\begin{equation*}
 u_h  \coloneqq  u_{a,h } + u_{j, h } + u_{c,h } \quad \text{for all } h > 0.
\end{equation*}
We clearly have $\{u_h\}_{h}$ in $SBV^2(a,b) $ for all $h >0$ and $u_h \to u$ in $L^1(a,b)$, which implies $Du_h \to Du$ in the flat norm.

We next address the modification of $\gamma$. We extend the absolutely continuous part~$g$ outside of $(a,b)$ by~$0$ and set $g_{a,h } \coloneqq  g * \psi_{h }$ for all $h >0$.
Then we have $g_{a,h } \in C^{\infty}(\R)$ for all $h >0$ and 
\begin{align}
g_{ a, h } \to g \text{ in } L^1(a,b). \label{absolutelygammacontinequalty}
\end{align}
Because of  $u_a^{\prime} -  g \in L^2(a,b)$ and $u_{a,h }^{\prime} -  g_{ a, h } = (u_a^{\prime} -  g) * \psi_{h }$ we further notice
\begin{align}
u_{a,h }^{\prime} -  g_{a, h } \to u_a^{\prime} - g \text{ in } L^2(a,b). \label{absolutelycontinequalty}
\end{align}
Now, we set 
\begin{equation*}
\gamma_h: = Du_{j,h }+ g_{h } \mathcal{L}^1 \quad \text{with } g_{h }  \coloneqq  g_{a, h } + u_{c, h }^{\prime}  \quad \text{for all } h > 0.
\end{equation*}
With the convergences $Du_h \to Du$ and $g_{a,h} - u_{a,h}' \to g - u_a'$ in the flat norm (via~\eqref{absolutelycontinequalty}), we infer $\gamma_h \to \gamma$ in the flat norm.
Since $u_{c, h }^{\prime}$ is canceled in the first term, we deduce
\begin{multline*}
\int_a^b  | u_h^{\prime} - g_{h }|^2 \dd x + c_0 \int_a^b|g_{h }| \dd x  + 2 \sum_{x \in J_{u_{h }} }\ff {\frac{1}{2}|[u_{h }](x)|} \\
\leq  \int_a^b  | u_{a,h }^{\prime} - g_{a, h }|^2 \dd x + c_0 \int_a^b \big( |g_{ a, h }| + | u_{c, h }^{\prime}| \big) \dd x + 2 \sum_{x \in J_{u_{j, h }}} \ff {\frac{1}{2}|u_{j, h }](x)|},
\end{multline*}
which, via~\eqref{jumpinequalty}, \eqref{cantorinequalty}, \eqref{absolutelygammacontinequalty} and~\eqref{absolutelycontinequalty}, implies 
\begin{equation}
\limsup_{h \to \infty} \bigg [ \int_a^b  | u_h^{\prime} - g_{h }|^2 \dd x + c_0 \int_a^b|g_{h }| \dd x  + 2 \sum_{x \in J_{u_{h }} }\ff {\frac{1}{2}|[u_{h }](x)|} \bigg] \leq E(u,\gamma).
  \label{estimategluedfunction}
\end{equation}
This does not yet show~\eqref{recovery_approximation}, since $\norm{u_h}_{L^{\infty}(a,b)} \leq K$ might not be satisfied for all $h >0$. We resolve this problem in two steps. With $\norm{u}_{L^{\infty}(a,b)} \leq K$ and $u_h \to u$ in $L^1(a,b)$, we can fix a sequence $\{\eta_{h }\}_{h }$ in $\R^+$ with $\eta_{h } \to 0^+$ as $h  \to 0$ and
\begin{equation}
|\lbrace x \in (a,b) \colon |u_h(x)| \geq K + \eta_{h } \rbrace | \to 0 \quad \text{as  } \,  h  \to 0. \label{etaarbi}
\end{equation}
We next define the truncated versions
\begin{equation*}
\tilde{u}_h(x)  \coloneqq  \min \{ \max \{ u_h (x), -K - \eta_{h } \} , K + \eta_{h } \} \quad \text{for all } h>0.
\end{equation*}
We then have $\tilde{u}_h \to u$ in $L^1(a,b)$, $D\tilde{u}_h \to Du$ in the flat norm and, in addition, also $\norm{\tilde{u}_h}_{L^{\infty}(a,b)} \leq K + \eta_{h }$ for all $h >0$. Correspondingly, we set
\begin{equation*}
\tilde{\gamma}_h  \coloneqq  D^j\tilde{u}_h + g_{h } \1_{ \lbrace\tilde{u}_h = u_h \rbrace}\mathcal{L}^1 \quad \text{for all } h>0.
\end{equation*}
By using~\eqref{eq:flat_trivial_estimate} and by applying subsequently the Cauchy--Schwarz inequality, we get
\begin{align*}
\norm{\tilde{\gamma}_{h } - \gamma_h}_{\textnormal{flat}} & \leq  \norm{(\tilde{\gamma}_{h } - D\tilde{u}_h) - (\gamma_h - Du_h)}_{\textnormal{flat}} + \norm{D\tilde{u}_h- Du_h }_{\textnormal{flat}}\\
 & =  \norm{(g_{h } - u_h^{\prime}) \1_{ \lbrace\tilde{u}_h \neq u_h \rbrace} \mathcal{L}^1}_{\textnormal{flat}} + \norm{D\tilde{u}_h- Du_h }_{\textnormal{flat}}\\
 & =  \norm{(g_{a,h } - u_{a, h }^{\prime})\1_{ \lbrace\tilde{u}_h \neq u_h \rbrace} \mathcal{L}^1}_{\textnormal{flat}} + \norm{D\tilde{u}_h- Du_h }_{\textnormal{flat}}\\
 & \leq \norm{(g_{a,h } - u_{a,h }^{\prime})\1_{ \lbrace\tilde{u}_h \neq u_h \rbrace} }_{L^1(a,b)}  + \norm{\tilde{u}_h- u_h}_{L^1(a,b)}\\
 & \leq \norm{g_{a,h } - u_{a,h }^{\prime}}_{L^2(a,b)} |\lbrace x \in (a,b) \colon \tilde{u}_h \neq u_h \rbrace|^{\frac{1}{2}} +  \norm{\tilde{u}_h- u_h}_{L^1(a,b)}.
\end{align*}
If we pass to the limit $h \to 0$ on the right-hand side, the first term vanishes because of the uniform boundedness of $u_h^{\prime} - g_h = u_{a,h}^{\prime} - g_{ a, h }$ in $L^2(a,b)$ due to~\eqref{absolutelycontinequalty} combined with the convergence 
\begin{align*}
\lbrace x \in (a,b) \colon \tilde{u}_h(x) \neq u_h \rbrace = \lbrace x \in (a,b) \colon |u_h(x)| > K + \eta_{h }\rbrace \to 0 \quad \text{as } h  \to 0
\end{align*} 
as a consequence from~\eqref{etaarbi}. Since with $u_h \to u$ and $\tilde{u}_h \to u$ in $L^1(a,b)$ also the second term vanishes, we conclude that $\tilde{\gamma}_{h } - \gamma_h \to 0$ in the flat norm. Consequently, we have established $\tilde{\gamma}_{h } \to \gamma$ in the flat norm. For $h  >0$ we finally define 
\begin{align*}
\hat{u}_{h }(x) \coloneqq  \frac{K}{K + \eta_{h }} \tilde{u}_h \quad \text{and} \quad \hat{\gamma}_{h }(x) \coloneqq  \frac{K}{K + \eta_{h }} \tilde{\gamma}_h(x) \quad  \text{for }  x \in (a,b).
\end{align*}
We clearly have $\norm{\hat{u}_{h }}_{L^{\infty}(a,b)} \leq K$, $J_{u_{h }} \subset J_{\hat{u}_{h }}$ and $|[\hat{u}_{h}]| \leq |[u_{h}]|$ for all $h  >0$. In view of $K /( K + \eta_{h }) \to 1$, we also have $\hat{u}_{h } \to u$ in $L^1(a,b)$ and $\hat{\gamma}_{h } \to \gamma$ in the flat norm. Moreover, if we denote the density of $\hat{\gamma}_{h }$ with respect to~$\mathcal{L}^1$ by $\hat{g}_{h }$, we observe $|\hat{u}_{h }^{\prime} - \hat{g}_{h}| \leq |u_h^{\prime} - g_{h }|$ and $|\hat{g}_{h}| \leq |g_{h }|$ on $(a,b)$ for all $h>0$. This shows, that the energy $E(\hat{u}_{h },\hat{\gamma}_{h })$ is finite for all $h>0$, with
\begin{align*}
 E(\hat{u}_{h },\hat{\gamma}_{h }) 
 & = \int_a^b  | \hat{u}_{h }^{\prime} - \hat{g}_{h }|^2 \dd x + c_0 \int_a^b|\hat{g}_{h }| \dd x \notag  + 2 \sum_{x \in J_{\hat{u}_{h }}} \ff {\frac{1}{2}|[\hat{u}_{h}](x)|} \\ & \leq \int_a^b  | u_h^{\prime} - g_{h }|^2 \dd x + c_0 \int_a^b|g_{h }| \dd x  + 2 \sum_{x \in J_{u_{h }}} \ff {\frac{1}{2}|[u_{h}](x)|}. 
\end{align*}
By taking into account~\eqref{estimategluedfunction}, we then obtain the claim~\eqref{recovery_approximation} (even for the $\limsup$), which ends the proof.
\end{proof}

\begin{remark}\label{rmk:SBV-relax-bv}
The function~$\hat{u}_h$ can even be chosen such that $\hat{u}_h(x) = u(a+)$ on $(a, \delta_h)$ and $\hat{u}_h(x) = u(b-)$ on $(b - \delta_h, b)$ for a sequence $\delta_h \searrow 0$. To see this, note that from $\hat{u}_h \to u$ a.e.\ and $\lim_{x \searrow a} u(x) = u(a+)$, $\lim_{x \nearrow b} u(x) = u(b-)$, one finds $\delta_h \searrow 0$ such that $a + \delta_h$ and $b - \delta_h$ are not contained in $J_{u_h}$ and $\lim_{h \to \infty} u_h(a+\delta_h) = u(a+)$, $\lim_{h \to \infty} u_h(b-\delta_h) = u(b-)$. Now consider $\hat{\hat{u}}_h \in SBV^2(a,b)$ defined by 
\[ \hat{\hat{u}}_h(x) = 
   \begin{cases} 
      u(a+)  &\text{for } x \in (a, a + \delta_h], \\ 
      \hat{u}_h(x) &\text{for } x \in (a + \delta_h, b - \delta_h), \\ 
      u(b-)  &\text{for } x \in [b - \delta_h, b) \\ 
   \end{cases} 
\]
and 
\[ \hat{\hat{\gamma}}_h 
   = \hat{\gamma}_h \mrs (a + \delta_h, b - \delta_h) 
     + \big( u_h(a + \delta_h) - u(a+) \big) \delta_{a + \delta_h} 
     + \big( u(b-) - u_h(b - \delta_h) \big) \delta_{b - \delta_h} \] 
The claim follows from observing that 
\[ E(\hat{\hat{u}}_h, \hat{\hat{\gamma}}_h) 
   \le E(\hat{u}_h, \hat{\gamma}_h) 
       + 2 c_0 | u_h(a + \delta_h) - u(a+) | 
       + 2 c_0 | u(b-) - u_h(b - \delta_h) | \] 
where the last two terms on the right-hand side vanish as $h \to \infty$ by construction. 
\end{remark}

Again, the case $K = \infty$ is a direct consequence.  
\begin{corollary}
For every $(u,\gamma) \in L^0((a, b);\overline{\R}) \times \M$ there is a sequence $\{(u_\varepsilon, \gamma_\varepsilon)\}_\varepsilon$ converging to $(u, \gamma)$ in $L^0((a, b);\overline{\R}) \times \M$ such that 
\begin{equation*}
  \limsup_{\varepsilon \to 0} E_\varepsilon(u_\varepsilon, \gamma_\varepsilon) 
  \le E(u,\gamma).
\end{equation*}
\end{corollary}

\begin{proof}
Without loss of generality we assume $u \in BV_{\infty,\mathcal{P}}(a,b)$. Let $a = x_0 < \ldots < x_m = b$ such that $\{x_1, \ldots, x_m\} = \{x \in (a, b) : |[u](x)| = \infty\}$. Note $\| u \|_{L^\infty(\mathcal{F}(u))} < \infty$. With the help of Theorem~\ref{thm:Gmma-limsup} we choose $\tilde{u}_\varepsilon, \tilde{g}_\varepsilon : (a,b) \to \R$ such that $\big( \tilde{u}_\varepsilon|_{(x_{i-1}, x_i)}, \tilde{g}_\varepsilon \mathcal{L}^1 \mrs (x_{i-1}, x_i) \big)$ is a recovery sequence for $\big( u|_{(x_{i-1}, x_i)}, \gamma \mrs (x_{i-1}, x_i) \big)$ whenever $(x_{i-1}, x_i) \subset \mathcal{F}(u)$ and $\tilde{u}_\varepsilon \equiv \pm \varepsilon^{-1}$, $\tilde{g}_\varepsilon \equiv 0$ on $(x_{i-1}, x_i)$ in case $u \equiv \pm \infty$ on $(x_{i-1}, x_i)$. 

As in the proof of Proposition~\ref{SBVupperbpundestimate} now define $u_{\varepsilon} \in W^{1,1}(a,b)$ by linear interpolation near the $x_i$ as
\[
\begin{array}{lll}
& u_{\varepsilon}  \coloneqq \tilde{u}_{\varepsilon} & \text{on } (a,b) \setminus \bigcup_{i=1}^{m-1} (x_i - \varepsilon^2 - 2 \varepsilon, x_i +\varepsilon^2 + 2 \varepsilon),\\
& u^{\prime}_{\varepsilon}  \coloneqq  \frac{\tilde{u}_{\varepsilon}(x_i + \varepsilon^2 + 2\varepsilon)-\tilde{u}_{\varepsilon}(x_i- \varepsilon^2 -2 \varepsilon)}{2\varepsilon^2} &\text{on } (x_i -\varepsilon^2 , x_i +\varepsilon^2) \text{ for } i \in \{1,\ldots,m-1\},\\[0.1cm]
& u^{\prime}_{\varepsilon}  \coloneqq  0 & \text{otherwise}.
\end{array}
\]
Clearly, $u_{\varepsilon} \to u$ a.e.~on $(a,b)$. Accordingly we define $\gamma_{\varepsilon} \in \mathcal{M}(a,b)$ as $\gae = g_{\varepsilon} \mathcal{L}^1$ with 
\[g_{\varepsilon}: =\begin{cases} \tilde{g}_{\varepsilon} & \text{on } (a,b) \setminus \bigcup_{i=1}^{m-1} (x_i - \varepsilon^2 - 2\varepsilon, x_i +\varepsilon^2 + 2\varepsilon),\\ 
\ue^{\prime}& \text{on } (x_i -\varepsilon^2 , x_i +\varepsilon^2) \text{ for } i \in \{1, \ldots,m-1\},
 \\0& \text{otherwise}.
\end{cases}\]
So still $\gae \to \gamma$ locally in the flat norm on each $(x_{i-1}, x_i) \subset \mathcal{F}(u)$. Estimating the energy as 
\begin{align*}
 E_{\varepsilon}(u_{\varepsilon}, \gamma_{\varepsilon})
 & \leq \int_{\mathcal{F}(u)} |\tilde{u}_\varepsilon' - \tilde{g}_\varepsilon|^2 \dd x + \frac{1}{\varepsilon} \int_{\mathcal{F}(u)} f \bigg(  \varepsilon \dashint_{I_{\varepsilon}(x)\cap (a,b)}{ |\tilde{g}_{\varepsilon}| \dd t }\bigg) \dd x \\ 
 & \quad + \frac{1}{\varepsilon} \int_{\bigcup_{i=1}^{m-1} (x_i - \varepsilon^2 -  \varepsilon, x_i +\varepsilon^2 + \varepsilon)} f\bigg(  \frac{1}{2} \int_{(x_i-\varepsilon^2,x_i+\varepsilon^2)} |u_{\varepsilon}^{\prime} | \dd t \bigg) \dd x, 
\end{align*}
where the last term on the right-hand side is bounded by 
$2(1+\varepsilon) (m-1) c_0 = 2(1+\varepsilon) c_0 \# J_u \setminus \mathcal{F}(u)$, we find that indeed 
\[ \limsup_{\varepsilon\to\infty} E_{\varepsilon}(u_{\varepsilon}, \gamma_{\varepsilon})
 \leq E(u,\gamma). \qedhere \] 
\end{proof}


\section[{\texorpdfstring{$\Gamma$}{Gamma}-convergence for the minimal energies with respect to \texorpdfstring{$\gamma$}{gamma}}]{\boldmath \texorpdfstring{$\Gamma$}{Gamma}-convergence for the minimal energies with respect to \boldmath \texorpdfstring{$\gamma$}{gamma}}
\label{sec: minimal energies 1d}

We finally prove the $\Gamma$-convergence result in Corollary~\ref{mainresult-K-infty} for the minimal energies with respect to the second variable $\gamma$, i.e., we consider the energies $\tilde{E}_{\varepsilon}$ and $\tilde{E}$ from~\eqref{def_minimal_energy_eps} and~\eqref{def_minimal_energy_limit}, respectively. We only treat the case $K < \infty$, the necessary modifications for $K = \infty$ are straightforward. Notice that, as a direct consequence of the fact that the function~$g^\ast$ from~\eqref{gammaaopt} solves the optimization problem in~\eqref{optimization_problem}, for every $u \in BV(a,b)$ there holds
\begin{equation}\label{optgamma}
\tilde{E}(u)=E(u, D^su+ g^\ast \mathcal{L}^1) = E(u,\go).
\end{equation}
For completeness we state also the corresponding compactness result.

\begin{corollary}[Compactness of the minimal energies with respect to $\gamma$]\label{comME}
Let $\{u_{\varepsilon}\}_{\varepsilon}$ be a sequence in $\LE$ with
\begin{equation*}
\tilde{E}_{\varepsilon}(\ue)\leq C_0 
\end{equation*}
for a positive constant $C_0$ and all $\varepsilon>0$. There exists a function $u \in BV(a,b)$ with $\norm{u}_{\LU} \leq K$ such that, up to a subsequence, $\{u_{\varepsilon}\}_{\varepsilon}$ converges to~$u$ in $\LE$.
\end{corollary}

\begin{proof}
We choose a low energy sequence $\{\gae\}_{\varepsilon}$ in $\M$ with $E_{\varepsilon}(\ue,\gae) \leq \tilde{E}_{\varepsilon}(\ue) + 1$ for all~$\varepsilon$. Since there holds $E_{\varepsilon}(\ue,\gae) \leq C_0 +1$ for all $\varepsilon$, according to Theorem~\ref{com} there exists a function $u \in BV(a,b)$ with $\norm{u}_{\LU} \leq K$ such that $\ue \to u$ in $L^1(a,b)$.
\end{proof}

\begin{proof}[Proof of Corollary~\ref{mainresult-K-infty}]
It is again sufficient to establish the $\Gamma$-$\liminf$ inequality and the ${\Gamma\text{-}\limsup}$-inequality only for $u \in BV(a,b)$ with $\norm{u}_{L^{\infty}(a,b)} \leq K$ since the estimates are trivial otherwise.

We first show the $\Gamma$-$\liminf$ inequality. We consider an arbitrary sequence $\{\ue\}_{\varepsilon}$ in $L^1(a,b)$ with $\ue \to u$ in $L^1(a,b)$, for which we may assume $\tilde{E}_{\varepsilon}(\ue)\leq C_0$ for some positive constant~$C_0$  and all~$\varepsilon$. We then select a low energy sequence $\{\gae\}_{\varepsilon}$ in $\M$ with
\begin{equation*}
E_{\varepsilon}(\ue,\gae) \leq \tilde{E}_{\varepsilon}(\ue) + \varepsilon \quad \text{ for every } \varepsilon>0.
\end{equation*}
By passing to a subsequence if necessary, we may assume that
\begin{equation*}
\liminf_{\varepsilon \to 0} E_{\varepsilon}(\ue,\gae) = \lim_{\varepsilon \to 0} E_{\varepsilon}(\ue,\gae).
\end{equation*}
At this stage we employ the compactness result of Theorem~\ref{com}: since $\ue \to u$ in $\LE$, there exists a function $g \in L^1(a,b)$ with $u^{\prime} - g \in L^2(a,b)$ such that, up to a subsequence, $\gae \to D^su + g \mathcal{L}^1$ in the flat norm. Since by Theorem~\ref{mainresult} we have $\Gamma$-convergence of $\{E_{\varepsilon}\}_{\varepsilon>0}$ to~$E$ in $L^1(a,b) \times \mathcal{M}(a,b)$ also for every subsequence, we obtain
\begin{align*}
\lim_{\varepsilon \to 0} E_{\varepsilon}(\ue,\gae) = \liminf_{\varepsilon \to 0} E_{\varepsilon}(\ue,\gae) \geq E(u, D^su + g \mathcal{L}^1),
\end{align*} 
which, by the choice of the sequence $\{\gae\}_{\varepsilon}$, shows  
\begin{equation*}
\liminf_{\varepsilon \to 0}\tilde{E}_{\varepsilon}(u_\varepsilon)\geq \tilde{E}(u).
\end{equation*}

We next show the $\Gamma$-$\limsup$ inequality. Via Theorem~\ref{mainresult}~(ii) we find a recovery sequence $\{(\ue,\gae)\}_{\varepsilon}$ of $(u,\go)$ in $L^{1}(a,b) \times \M$. In combination with~\eqref{optgamma} this yields the claim
\begin{equation*}
\limsup_{\varepsilon \to 0} \tilde{E}_{\varepsilon}(\ue)\leq \limsup_{\varepsilon \to 0} E_{\varepsilon}(\ue,\gae) \leq E(u,\go) = \tilde{E}(u). \qedhere
\end{equation*}
\end{proof}


\section*{Acknowledgments} 
V.~A.-V.\ gratefully acknowledges financial support of the Konrad-Adenauer-Stiftung. All authors are grateful to the referee for valuable suggestions that have led to substantial improvements of the results.



\end{document}